\documentclass{amsart}
\usepackage{amsmath}
\usepackage{amsfonts}
\usepackage{graphics}
\usepackage{epsfig}
\usepackage{amssymb}
\usepackage{amscd}
\usepackage[all]{xy}
\usepackage{latexsym}
\usepackage{graphicx}
\usepackage{multirow}
\usepackage{geometry}
\usepackage{color}
\usepackage[usenames,dvipsnames]{pstricks}
\usepackage{epsfig}
\usepackage{pst-grad} 
\usepackage{pst-plot}

\newtheorem{thm}{Theorem}[section]

\newtheorem{prop}[thm]{Proposition}

\theoremstyle{definition}
\newtheorem{Def}[thm]{Definition}
\newtheorem{rem}[thm]{Remark}
\newtheorem*{ack}{Acknowledgement}

\newtheorem{case}{Case}

\numberwithin{equation}{section}
\numberwithin{figure}{section}

\def\tr{{\text{\rm{tr}}}}
\def\rchi{{\hbox{\raise1.5pt\hbox{$\chi$}}}}

\def\isom{\cong}
\def\tensor{\otimes}
\def\dsum{\oplus}

\def\const{{\text{\rm{const}}}}

\def\a{\alpha}
\def\b{\beta}
\def\lam{\lambda}

\def\gam{\gamma}
\def\Gam{\Gamma}

\def\Sym{{\text{\rm{Sym}}}}

\def\Spec{{\text{\rm{Spec}}}}
\def\Arg{{\text{\rm{Arg}}}}
\def\Proj{{\text{\rm{Proj}}}}
\def\supp{{\text{\rm{supp}}}}
\def\Airy{{\text{\rm{Airy}}}}
\def\Catalan{{\text{\rm{Catalan}}}}
\def\Pic{{\text{\rm{Pic}}}}
\def\Gal{{\text{\rm{Gal}}}}
\def\NS{{\text{\rm{NS}}}}
\def\erf{{\text{\rm{erf}}}}
\def\Gauss{{\text{\rm{Gau\ss}}}}

\newcommand{\bea}{\begin{eqnarray}}
\newcommand{\eea}{\end{eqnarray}}
\newcommand{\be}{\begin{equation}}
\newcommand{\ee}{\end{equation}}

\newcommand{\Mbar}{{\overline{\mathcal{M}}}}
\newcommand{\bA}{{\mathbb{A}}}
\newcommand{\bP}{{\mathbb{P}}}
\newcommand{\bC}{{\mathbb{C}}}

\newcommand{\bF}{{\mathbb{F}}}

\newcommand{\bR}{{\mathbb{R}}}

\newcommand{\bZ}{{\mathbb{Z}}}

\newcommand{\cM}{{\mathcal{M}}}

\newcommand{\cD}{{\mathcal{D}}}
\newcommand{\cE}{{\mathcal{E}}}

\newcommand{\cO}{{\mathcal{O}}}

\newcommand{\cU}{{\mathcal{U}}}

\newcommand{\la}{{\langle}}
\newcommand{\ra}{{\rangle}}
\newcommand{\half}{{\frac{1}{2}}}

\newcommand{\rar}{\rightarrow}
\newcommand{\lrar}{\longrightarrow}

\begin{document}
\large
\setcounter{section}{0}

\allowdisplaybreaks

\title[Quantization of 
meromorphic Higgs bundles]
{Quantization of spectral  curves for 
meromorphic Higgs bundles through
topological recursion}

\author[O.\ Dumitrescu]{Olivia Dumitrescu}
\address{
Institut f\"ur Algebraische Geometrie\\
Fakult\"at f\"ur Mathematik und Physik\\
Leibniz Universit\"at Hannover\\
Welfengarten 1\\
30167 Hannover, Germany}
\email{dumitrescu@math.uni-hannover.de}

\author[M.\ Mulase]{Motohico Mulase}
\address{
Department of Mathematics\\
University of California\\
Davis, CA 95616--8633, U.S.A.}
\email{mulase@math.ucdavis.edu}

\begin{abstract}
A geometric quantization  using
the topological recursion is established for 
the compactified cotangent bundle
of a smooth projective curve
of an arbitrary genus. In this quantization,
the Hitchin spectral curve of a rank $2$ 
meromorphic Higgs bundle
on the base curve corresponds to 
a quantum curve, which is
   a Rees $D$-module on the base. 
   The topological recursion then gives an all-order
 asymptotic expansion of its solution,
 thus determining a state vector corresponding
 to the spectral curve as a meromorphic 
 Lagrangian. 
  We establish a generalization of the
topological recursion for a singular 
spectral curve.
We show that
the partial differential equation
version of the topological recursion  
automatically selects the normal ordering 
of the canonical coordinates, and determines the
unique quantization of the spectral curve.
The quantum curve thus
constructed has the semi-classical limit that
agrees with the original  spectral curve.
Typical examples of our construction
includes classical differential equations, such 
as Airy, Hermite, and Gau\ss\ hypergeometric
 equations. The topological recursion
 gives an asymptotic expansion of solutions to
these equations at their singular points, 
relating  Higgs bundles
and various quantum
invariants. 
\end{abstract}

\thanks{The first author is a member of 
the Simion Stoilow Institute of Mathematics of the 
Romanian Academy}



\maketitle

\tableofcontents

\section{Introduction}
\label{sect:intro}

\subsection{Overview}

The \textbf{topological recursion} of
\cite{EO1} was originally conceived as 
a computational mechanism to find the multi-resolvent
correlation functions of random matrices
\cite{CEO, E2004}. It 
 has been proposed that the topological 
 recursion is an effective
tool for defining a genus $g$ B-model
topological string theory 
 on a holomorphic curve (known as an
\textbf{Eynard-Orantin spectral curve}),
that should be the
mirror symmetric dual to the genus $g$
Gromov-Witten theory on the A-model side
\cite{BKMP, BKMP2, Marino}. 
This  correspondence
has been rigorously established   for several examples, 
most notably for an arbitrary
toric Calabi-Yau orbifold of  $3$ 
dimensions \cite{FLZ},
and many other enumerative geometry problems
\cite{BHLM, DoMan, DMSS, EMS, EO3, MZ}.

\textbf{Quantum curves} are introduced in
the physics literature (see for example,
\cite{ADKMV, DHS, DHSV, GS, Hollands})
as a device  to compactly encode 
the information of quantum invariants arising
in Gromov-Witten theory, Seiberg-Witten theory,
and knot theory. The \emph{semi-classical
limit} of a quantum curve is a holomorphic curve 
defining a B-model that is mirror dual
to the A-model for these  quantum invariants. 
Geometrically, a quantum curve also appears as
an $\hbar$-deformation of a  generalized
 Gau\ss-Manin  connection (or Picard-Fuchs
 differential equation)
on a curve, with regular and 
\emph{irregular} singularities.

Since both quantum curves and the topological 
recursion produce B-models  on a holomorphic
curve, it is natural to ask if they are related. 
Indeed, it was proposed 
by physicists
 \cite{DFM, GS} for the context of knot theory
 that the topological recursion would give a
 perturbative construction of 
 quantum curves. 
 So far such a relation is
 not fully understood
  in the  mathematical 
 examples of quantum curves
 constructed in \cite{BHLM, DMNPS,
 MSS, MS}.

The purpose of this paper is to establish a clear
geometric relation between quantum curves
and topological recursion for the
 \textbf{Hitchin spectral curves}
associated with Higgs bundles on a base curve 
$C$, with  
arbitrary \emph{meromorphic} Higgs fields.
Although the language of 
\textbf{geometric quantization} does not work
in this  algebraic geometry context, let us use it
for a moment as an analogy. Then 
the main result of this paper could be
understood as follows: \emph{the topological recursion
is a geometric quantization of}  
 $T^*C$. A Hitchin spectral curve is 
a (\emph{meromorphic}) Lagrangian in the
holomorphic symplectic manifold $T^*C$. Using
the topological recursion, we construct a
\emph{state vector}, which is a solution to 
the  Schr\"odinger equation on $C$ that is uniquely
determined by the spectral curve. 
The state vector is 
equivalent to a quantum curve in our setting,
as a Rees $D$-module on $C$. 
More precisely, 
we prove the following.

\begin{thm}[Main results]
Let $C$ be a smooth projective
curve of an arbitrary genus, and $(E,\phi)$ a
Higgs bundle of rank $2$  on $C$
with a meromorphic
Higgs field $\phi$. Denote by 
\be
\label{compact T*C}
\overline{T^*C}:= \bP(K_C\dsum \cO_C)
\overset{\pi}{\lrar} C
\ee
the compactified cotangent bundle of $C$
(see \cite{KS2013}),
which is a ruled surface on the base $C$.
Here, $K_C$ is the canonical sheaf.
The Hitchin spectral curve 
\begin{equation}
\label{spectral}
\xymatrix{
\Sigma \ar[dr]_{\pi}\ar[r]^{i} 
&\overline{T^*C}\ar[d]^{\pi}
\\
&C		}
\end{equation} 
 for a meromorphic
Higgs bundle is defined as the divisor 
of zeros on $\overline{T^*C}$
of the characteristic polynomial of
$\phi$:
\be
\label{char poly}
\Sigma = \left(\det(\eta - \pi^*\phi)\right)_0,
\ee
where $\eta \in H^0(T^*C, \pi^*K_C)$ is the 
tautological $1$-form on $T^*C$ extended as
a meromorphic $1$-form on the
compactification $\overline{T^*C}$.

\begin{itemize}
\item The integral topological recursion  of 
\cite{DM2014,EO1} is extended to the
 curve $\Sigma$, as \eqref{integral TR}. 
For this purpose,
we blow up $\overline{T^*C}$
several times
as in \eqref{blow-up} to construct the normalization
$\widetilde{\Sigma}$.
The construction of $Bl(\overline{T^*C})$ is
given in Definition~\ref{def:Bl}. It is the minimal 
resolution of the support $ \Sigma \cup C_\infty$
of the
\emph{total} divisor
 \be
 \label{total}
  \Sigma - 
2C_\infty = \left(\det(\eta - \pi^*\phi)\right)_0
-\left(\det(\eta - \pi^*\phi)\right)_\infty
 \ee
 of the characteristic polynomial, where
\begin{equation}
\label{C-infinity}
C_\infty := \bP(K_C\dsum\{0\}) = 
\overline{T^*C}\setminus T^*C
\end{equation}
is the divisor at infinity. Therefore,
in  $Bl(\overline{T^*C})$, 
the proper transform $\widetilde{\Sigma}$
of $\Sigma$
is smooth and does not intersect
with
the proper transform of $C_\infty$.
\begin{equation}
 \label{blow-up}
\xymatrix{
\widetilde{\Sigma} 
\ar[dd]_{\tilde{\pi}} \ar[rr]^{\tilde{i}}\ar[dr]^{\nu}&&Bl(\overline{T^*C})
\ar[dr]^{\nu}
\\
&\Sigma \ar[dl]_{\pi}\ar[rr]^{i} &&
\overline{T^*C}  \ar[dlll]^{\pi}
\\
C 		}
\end{equation} 

\item The genus of the normalization 
$\widetilde{\Sigma}$ is given by
$$
g(\widetilde{\Sigma}) = 2g(C)-1+\half \delta,
$$
where $\delta$ is the sum of the  number of cusp singularities
of $\Sigma$ 
and the ramification points 
of $\pi:\Sigma\lrar C$ (Theorem~\ref{thm:geometric genus formula}).

\item  The topological recursion 
thus generalized requires 
a globally defined meromorphic $1$-form $W_{0,1}$
on $\widetilde{\Sigma}$ and 
a symmetric meromorphic $2$-form $W_{0,2}$
on the product $\widetilde{\Sigma}
\times \widetilde{\Sigma}$ as the initial data.
We choose
\begin{equation}
\begin{cases}
\label{W0102}
W_{0,1} = \tilde{i}^*\nu^*\eta\\
W_{0,2} = d_1d_2 \log E_{\widetilde{\Sigma}},
\end{cases}
\end{equation}
where
$E_{\widetilde{\Sigma}}$ is a normalized 
Riemann prime form on 
${\widetilde{\Sigma}}$
(see \cite[Section~2]{DM2014}). The form $W_{0,2}$
depends only on the intrinsic geometry of 
the smooth curve $\widetilde{\Sigma}$. The
geometry of \eqref{blow-up}
is  encoded in $W_{0,1}$. 
The integral topological recursion
 produces a symmetric meromorphic $n$-linear
differential form $W_{g,n}(z_1,\dots,z_n)$
on $\widetilde{\Sigma}$ for every $(g,n)$
subject to $2g-2+n>0$ from the initial data 
\eqref{W0102}.

\item The residue evaluation of the 
integral topological recursion \eqref{integral TR}
is explicitly 
performed as in \cite[(4.7)]{DM2014}, 
and we obtain a differential 
recursion \eqref{differential TR}. 
It determines the free energy $F_{g,n}(z_1,\dots,z_n)$,
a symmetric meromorphic function on $\cU^n$
for $2g-2+n>0$, up to a constant.
Here, $\varpi:\cU\lrar \widetilde{\Sigma}$ is
the universal covering of  $\widetilde{\Sigma}$.

\item The quantum curve associated with 
the Hitchin spectral curve $\Sigma$ is defined
as a Rees $D$-module  
(Definition~\ref{def:qc}) on $C$. On each 
coordinate neighborhood $U\subset C$ with
coordinate $x$, a generator of the quantum curve
is given by 
$$
P(x,\hbar) = \left(\hbar\frac{d}{dx}\right)^2
-\tr \, \phi(x) \left(\hbar\frac{d}{dx}\right)
+ \det \phi(x).
$$
In particular, the \textbf{semi-classical limit}
of the quantum curve recovers the singular
spectral curve $\Sigma$, not its normalization
$\widetilde{\Sigma}$.

\item We construct
the \textbf{all-order WKB expansion} 
\be
\label{WKB}
\Psi(x,\hbar) = \exp\left(\sum_{m=0}^\infty
\hbar^{m-1} S_m(x)\right)
\ee
of a solution to the 
Schr\"odinger equation
\be
\label{Sch}
\left(\left(\hbar\frac{d}{dx}\right)^2
-\tr \, \phi(x) \left(\hbar\frac{d}{dx}\right)
+ \det \phi(x))\right)
\Psi(x,\hbar)=0,
\ee
near  each critical value of 
${\pi}:{\Sigma}\lrar C$,
in terms of the free energies.
Indeed, \eqref{Sch}
is equivalent to the  \textbf{principal 
specialization} of the differential recusion
\eqref{differential TR}. The equivalence  
is given by
\be
\label{Sm in Fgn}
S_m(x) = \sum_{2g-2+n=m-1} \frac{1}{n!}
F_{g,n}\big(z(x)\big),
\ee
where $F_{g,n}\big(z(x)\big)$ is the principal 
specialization of $F_{g,n}(z_1,\dots,z_n)$
evaluated at a local section $z=z(x)$ of
$\tilde{\pi}:\widetilde{\Sigma}\lrar C$.

\item The canonical ordering of the quantization of
the local functions on $T^*C$ is automatically
chosen in the process of the integration 
from \eqref{integral TR} to \eqref{differential TR}
and the principal specialization \eqref{Sm in Fgn}.
This selects the canonical ordering in 
\eqref{Sch}.

\end{itemize}
\end{thm}

\begin{rem}
Although $\overline{T^*C}$ is not a holomorphic
symplectic manifold, in the analogy of
geometric quantization mentioned above,
our quantization is similar to a holomorphic
quantization of $T^*C$, where the fiber
coordinate is quantized to $\hbar\frac{d}{dx}$.
A Hitchin spectral curve is a \emph{meromorphic} 
Lagrangian, and corresponds via 
the topological recursion to a 
state vector $\Psi(x,\hbar)$ of
\eqref{WKB}.
\end{rem}

\begin{rem}
The constant ambiguity in the 
symmetric function $F_{g,n}$ is 
reflected in a factor $\exp(\sum \hbar^{m-1} c_m)$
multiplied to $\Psi(x,\hbar)$ of \eqref{WKB},
where $c_m$ is an arbitrary constant. 
Therefore, our method does not determine the
$\hbar$ dependence of the solution to 
\eqref{Sch}.
\end{rem}

\begin{rem}
The current paper is a 
generalization of \cite{DM2014}.
In the process of establishing a geometric
theory of topological recursion and quantum 
curves, we have discovered 
in \cite{DM2014} that the
topological recursion of \cite{EO1} can be naturally
generalized to the Hitchin spectral curves
for holomorphic 
Higgs bundles defined 
on a smooth projective  curve $C$ of 
genus $g(C)\ge 2$. We have then
showed   that the Hitchin spectral 
curve for an $SL(2,\bC)$-Higgs bundle
is quantizable, and that the topological 
recursion gives an asymptotic expansion of 
a holomorphic solution to the quantum curve
\eqref{Sch} with $\tr\, \phi = 0$.
\end{rem}

\begin{rem}
The singularities of the quantum curve, which are
  regular and 
irregular singular points 
of a differential equation
\eqref{Sch} on the base
curve $C$, are analyzed by the geometry of the 
Hitchin
spectral curve $\Sigma$ 
(Theorem~\ref{thm:qc singularity}). For example,
the number of resolutions required to desingularize
$\Sigma\cup C_\infty$ at $P\in \Sigma$ is
always $\lceil r\rceil$ if $\pi(P)$ is an irregular 
singular point of class $r-1$. 
\end{rem}

\begin{rem}
Already several mathematical 
examples of quantum curves have been
 rigorously constructed  for enumerative
geometry problems, such as Catalan numbers
and their generalizations,
simple and double Hurwitz numbers and their
variants, and 
 Gromov-Witten invariants of a point,
the projective line, and a few toric
Calabi-Yau threefolds \cite{BHLM, DoMan, 
DMNPS, MSS, MS, Zhou1, 
Zhou2}. In knot theory, a  
quantum curve is the same as a $q$-holonomic
operator $\widehat{A}$ that quantizes
the \textbf{A-polynomial} of a knot and characterizes
the corresponding colored Jones polynomial
\cite{Gar, GarLe}.
\end{rem}

\begin{rem}
Another aspect of quantum curves lies in 
its relation to \textbf{non-Abelian Hodge
correspondence}.
A quantum curve  is an 
$\hbar$-connection on the base curve $C$, and
the  Higgs field is
  recovered as its \emph{classical limit}
 $\hbar \rar 0$. 
 The non-Abelian Hodge
correspondence with irregular singular points
has been studied extensively  both in 
mathematics and physics, starting 
from the fundamental papers
\cite{BB, B2001} and to more recent ones, including
\cite{B2012,W2008, W2008-2}. 
 \end{rem}

Our current paper is motivated by the following 
simple question:
\emph{If quantum curves are truly fundamental 
objects, then where do we see them most commonly,
in particular, in classical mathematics?} 
The answer we propose in this paper is that 
the classical differential equations, such as
the Airy, Hermite, and Gau\ss\ hypergeometric
differential equations, are natural 
examples of our construction 
of the quantum curves that are associated
with stable 
meromorphic
Higgs bundles defined over the projective
line $\bP^1$. The topological recursion then
 gives an all-order asymptotic expansion of 
 their solutions, connecting Higgs bundles
  to  the world of quantum invariants.

  Once we study these concrete classical examples,
  it becomes plausible that the base curve $C$ of
  the Higgs bundle and a spectral 
  curve $\Sigma\subset \overline{T^*C}$
  are \emph{moduli spaces} of certain geometries. 
  For example, 
  in a particular case of the Gau\ss\ 
  hypergeometric equations considered
  in Sections~\ref{classical examples}
  and \ref{sub:Gauss}, the base curve
  is actually
  $\Mbar_{0,4}\isom \bP^1$. The spectral 
  curve for this example is the moduli space of
  elliptic curves, together with the two 
  eigenvalues of the 
  classical limit of 
    the  \textbf{Gau\ss-Manin connection} 
  \cite{Manin} that
  characterizes the 
  periods of elliptic curves. 
  
  More precisely,
  for every
  $x\in \cM_{0,4}$, we consider the elliptic
  curve $E(x)$ ramified over $\bP^1$ at four points
  $\{0,1,x,\infty\}$, and its two periods 
   given by the elliptic integrals
  \cite{KZ}
  \begin{equation}
\label{elliptic periods}
\omega_1(x)=
\int_1^\infty \frac{ds}{\sqrt{s(s-1)(s-x)}},
\qquad
\omega_2(x) = 
\int_x^1 \frac{ds}{\sqrt{s(s-1)(s-x)}}.
\end{equation}
The quantum curve in this case is 
  an $\hbar$  \textbf{-deformed 
Gau\ss-Manin connection}
\begin{equation}
\label{Gauss-Manin}
\nabla^\hbar_{GM} =\hbar d -
\begin{bmatrix}
&\frac{1}{x}\\ \\
-\frac{1}{4(x-1)}&-\frac{2x-1}{x(x-1)}
+\frac{\hbar}{x}
\end{bmatrix}dx
\end{equation}
in the trivial bundle $\cO_{\Mbar_{0,4}}\dsum
\cO_{\Mbar_{0,4}}$ of rank $2$
over $\Mbar_{0,4}$. Here, $d$ denotes
the exterior differentiation acting on the 
local sections of this trivial bundle. 
 The restriction $\nabla^1_{GM}$ 
 of the connection at $\hbar=1$ is equivalent to
the Gau\ss-Manin connection that
characterizes 
   the two periods of \eqref{elliptic periods},
   and the Higgs field is the classical limit
   of the connection matrix
   at $\hbar\rar 0$:
   \begin{equation}
   \label{Gauss-Manin-Higgs}
\phi =    \begin{bmatrix}
&\frac{1}{x}\\ \\
-\frac{1}{4(x-1)}&-\frac{2x-1}{x(x-1)}
\end{bmatrix}dx.
\end{equation}
  The spectral curve 
  $\Sigma \subset \overline{T^*\Mbar_{0,4}}$ 
  as a moduli space consists
  of the data 
  $\big(E(x), \a_1(x), \a_2(x)\big)$,
  where $\a_1(x)$ and $\a_2(x)$ are the two
  eigenvalues of the  Higgs field $\phi$. 
  The spectral curve 
  $\Sigma \subset \overline{T^*\Mbar_{0,4}} = \bF^2$ 
  as a divisor in the Hirzebruch surface 
  is  determined by the characteristic 
  equation 
  \be
  \label{char intro}
  y^2 + \frac{2x-1}{x(x-1)}y
  + \frac{1}{4x(x-1)} = 0
  \ee
  of the Higgs field. Geometrically,
  $\Sigma$ is a singular rational curve with one
  ordinary double point at $x=\infty$.
  As we see in the later sections, the quantum curve
  is a \textbf{quantization} of the characteristic equation
  \eqref{char intro}
  for the eigenvalues $\a_1(x)$ and $\a_2(x)$ 
 of $\phi(x)$. It is an $\hbar$-deformed
   \textbf{Picard-Fuchs equation}
   $$
   \left(\left(\hbar \frac{d}{dx}\right)^2 
   + \frac{2x-1}{x(x-1)} \left(\hbar \frac{d}{dx}\right)
   + \frac{1}{4x(x-1)}\right) \omega_i(x,\hbar) =0,
   $$
    and its semi-classical limit agrees with
   the singular spectral curve
  $\Sigma$. As a second order differential equation,
  the quantum curve has two independent
  solutions corresponding
  to the two eigenvalues. At $\hbar=1$, 
  these solutions are
  exactly the two periods $\omega_1(x)$ and
  $\omega_2(x)$ of the Legendre family of 
  elliptic curves $E(x)$. 
  The topological recursion 
  produces  asymptotic
  expansions of these periods  as
  functions in $x\in \Mbar_{0,4}$, at which the
  elliptic
  curve $E(x)$ degenerates to a nodal rational 
  curve.
  
  \begin{rem}
  Although we do not deal with quantum curves
  associated with knots (cf. \cite{GS}) in our
  current paper, there a
  spectral curve is the $SL(2,\bC)$-character variety 
  of the
  knot complement in the $3$-sphere $S^3$.
  Thus the spectral curve is again a moduli space,
  this time the moduli of
  flat $SL(2,\bC)$-connections on the knot 
  complement.
  \end{rem}

When we deal with a singular spectral curve 
$\Sigma\subset \overline{T^*C}$,
the key question
is how to relate the singular curve with smooth ones.
In terms of the Hitchin fibration, a 
singular spectral
curve corresponds to a degenerate Abelian variety
in the family.
There are two different approaches to 
this question:
\begin{enumerate}
\item Deform $\Sigma$ locally 
in the base of the Hitchin 
fibration to a family 
of non-singular
curves, and study the 
quantization associated with 
this deformation family.
\item Blow up  $\overline{T^*C}$ and obtain
the  resolution of singularities
$\widetilde{\Sigma}$
of the singular spectra curve $\Sigma$. 
Then construct the
quantum curve for $\Sigma$ using the geometry of
$\widetilde{\Sigma}$.
\end{enumerate}
In this paper we will pursue the second path, and 
give a construction of a quantum curve using the
geometric
information of the blow-up
\eqref{blow-up}.

In the Higgs bundle 
context, a quantum curve 
 is a Rees $D$-module over the Rees ring
$\widetilde{\cD_C}$
defined by the canonical filtration of $\cD_C$
(see for example, \cite{G}), such that its
semi-classical limit coincides with the
Hitchin spectral curve of a
meromorphic Higgs bundle on  $C$.
 Here, $\cD_C$
denotes the sheaf of linear ordinary differential
operators on $C$. A $\widetilde{\cD_C}$-module 
is a particular $\hbar$-deformation 
family of  $\cD_C$-modules.
Suppose a Rees $\widetilde{\cD_C}$-module
is written locally as
$$
\cM(U) = \widetilde{\cD_C}(U)\big/
\widetilde{\cD_C}(U)\cdot P(x,\hbar)
$$
on an open 
disc $U\subset C$ with a local coordinate $x$,
where
 $P(x,\hbar)\in \widetilde{\cD_C}(U)$ 
 is a linear ordinary differential operator 
 depending on the deformation
parameter $\hbar$. This operator
 then characterizes, by an equation
\begin{equation}
\label{PPsi=0}
P(x,\hbar) \Psi(x,\hbar) = 0,
\end{equation}
 the \emph{partition function} $\Psi(x,\hbar)$ of a 
topological quantum field theory 
on a `space' that is considered to be the \emph{mirror
dual} to the spectral curve.  
 The physics theories appearing
in this way are 
related to quantum topological 
invariants and  geometric enumeration
problems.
 The variable $x$ of the base curve $C$ 
is  usually 
the parameter of  generating functions of the
quantum invariants that are considered in the 
theory, and the generating functions determine
a particular asymptotic expansion
of an analytic solution $\Psi(x,\hbar)$ of
\eqref{PPsi=0} around its  singularity.

\subsection{Classical examples}
\label{classical examples}

Riemann and Poincar\'e 
worked on
the interplay between algebraic geometry 
of  curves in a ruled surface and 
the asymptotic expansion
   of an analytic solution 
   to a differential equation defined on the 
base curve of the ruled surface.  
The theme of the current paper lies exactly on this
link, looking at the classical subject from a 
new point of view.

Let us  recall the definition of 
regular and irregular singular points  of 
a second order differential
equation.

\begin{Def}
\label{def:regular and irregular} 
Let
\begin{equation}
\label{second}
\left(\frac{d^2}{dx^2}+a_1(x)\frac{d}{dx}+a_2(x)
\right)\Psi(x) = 0
\end{equation}
be a second order differential equation
defined around a neighborhood of $x=0$ on a
small disc $|x|< \epsilon$ with meromorphic 
coefficients $a_1(x)$ and $a_2(x)$ with poles  
at $x=0$. Denote by $k$ (reps.\ $\ell$) the 
order of the pole of 
$a_1(x)$ (resp.\ $a_2(x)$) at $x=0$. 
If $k \le 1$  and $\ell\le 2$, then \eqref{second}
has a \textbf{regular singular point} at $x=0$.
Otherwise, consider the \emph{Newton polygon}
 of the order of poles of the coefficients of
 \eqref{second}. It is the upper part of
 the convex hull of three
 points $(0,0), (1, k), (2,\ell)$. As a convention,
 if $a_j(x)$ is identically $0$, then
 we assign $-\infty$ as its pole order. Let $(1,r)$
 be the intersection point of
 the Newton polygon and the line $x=1$.
 Thus
 \begin{equation}
 \label{irregular class}
r= \begin{cases}
 k \qquad 2k\ge \ell,\\
\frac{\ell}{2} \qquad 2k\le \ell.
 \end{cases}
 \end{equation}
 The differential equation \eqref{second} 
 has an \textbf{irregular singular point of class}
 $r-1$ at $x=0$ if $r>1$.
\end{Def}

To illustrate the scope of interrelations among the
geometry of meromorphic 
Higgs bundles, their spectral curves, the singularities
of  quantum curves, $\hbar$-connections,
and the quantum invariants, let us tabulate 
five examples here (see Table~\ref{tab:examples}).
The differential operators of these equations 
are listed on the third column. 
In the first three rows, the quantum curves are
examples of 
classical differential equations known as 
 Airy,  Hermite,  the Gau\ss\ hypergeometric
equations. 
The fourth and the fifth
rows are added to show that it is \emph{not}
the singularity of the spectral curve that
determines the singularity
of the quantum curve.
In each example, the Higgs bundle $(E,\phi)$ 
we are 
considering consists of the base curve $C=\bP^1$
and the trivial vector bundle $E=\cO_{\bP^1}
\dsum \cO_{\bP^1}$ of rank $2$ on ${\bP^1}$.

\begin{table}[htb]
\label{tab:examples}
  \centering
  
  \begin{tabular}{|c|c|c|}

\hline 

Higgs Field & Spectral Curve  & Quantum Curve
\tabularnewline
\hline \hline
$\begin{bmatrix}
&1\\
x
\end{bmatrix}dx$ & 
$\begin{matrix}
y^2-x=0\\
w^2-u^5=0\\
\Sigma = 2C_0+5F\\
p_a=2,p_g=0
\end{matrix}
$ & 
$\begin{matrix}
\text{Airy}\\
\left(\hbar\frac{d}{dx}\right)^2 - x\\
\text{Class $\frac{3}{2}$ irregular singularity}\\
\text{at $\infty$}
\end{matrix}
$
 \tabularnewline
\hline 
$\begin{bmatrix}
&1\\
-1&-x
\end{bmatrix}dx$ & 
$\begin{matrix}
y^2+xy+1=0\\
w^2-uw+u^4=0\\
\Sigma = 2C_0+4F\\
p_a=1, p_g=0
\end{matrix}$ &
$\begin{matrix}
\text{Hermite}\\
\left(\hbar\frac{d}{dx}\right)^2 +x\hbar\frac{d}{dx}
+1\\
\text{Class $2$ irregular singularity}\\
\text{at $\infty$}
\end{matrix}$
  \tabularnewline
\hline 
 $\begin{bmatrix}
&\frac{1}{x}\\ \\
\frac{1}{4(1-x)}&\frac{2x-1}{x(1-x)}
\end{bmatrix}dx$ & 
$\begin{matrix}
y^2
+\frac{2x-1}{x(x-1)}y+\frac{1}{4x(x-1)}=0
\\
w^2+4(u-2)uw
\\
-4u^2(u-1)=0
\\
\Sigma = 2C_0+4F\\
p_a=1,p_g=0
\end{matrix}
$ &
$\begin{matrix}
\text{Gau\ss\ Hypergeometric}\\
\left(\hbar\frac{d}{dx}\right)^2
+\frac{2x-1}{x(x-1)}
\hbar\frac{d}{dx}+\frac{1}{4x(x-1)}\\
\text{Regular singular points}\\
\text{at $x=0,1,\infty$}
\end{matrix}
$
  \tabularnewline
\hline 

 $\begin{bmatrix}
&1\\
-\frac{1}{x+1}&-1
\end{bmatrix}dx$ & 
$
\begin{matrix}
y^2+y+\frac{1}{x+1}=0
\\
w^2-u(u+1)w\\
+u^3(u+1)=0
\\
\Sigma = 2C_0+4F\\
p_a=1,p_g=0
\end{matrix}
$ &
$\begin{matrix}
\left(\hbar\frac{d}{dx}\right)^2+ \hbar\frac{d}{dx}
+\frac{1}{x+1}\\
\text{Regular singular point at $x=-1$}\\
\text{and a class $1$ irregular singularity}\\
\text{at $x=\infty$}
\end{matrix}
$

  \tabularnewline
\hline 

$\begin{bmatrix}
&1\\
\frac{1}{x^2-1}&-\frac{2x^2}{x^2-1}
\end{bmatrix}dx$ & 
$
\begin{matrix}
(x^2-1)y^2+2x^2y-1=0
\\
\text{non-singular}
\\
\Sigma = 2C_0+4F\\
p_a=p_g=1
\end{matrix}
$ &
$\begin{matrix}
\left(\hbar\frac{d}{dx}\right)^2+
2\frac{x^2}{x^2-1} \hbar\frac{d}{dx}
-\frac{1}{x^2-1}\\
\text{Regular singular points at $x=\pm1$}\\
\text{and a class $1$ irregular singularity}\\
\text{at $x=\infty$}
\end{matrix}
$

  \tabularnewline
\hline 

\end{tabular}
\bigskip

  \caption{Examples of quantum curves.}
\end{table}

The first column of the table shows the Higgs field
$\phi:E\lrar E\tensor K_{\bP^1}(2)$. 
Here, $x$ is the affine coordinate
of $\bP^1\setminus \{\infty\}$. 
Since our vector bundle is trivial, the non-Abelian
Hodge correspondence is simple in each case. 
Except for the Gau\ss\ hypergeometric case, 
it is 
given by
\begin{equation}
\label{hbar-connection intro}
\nabla^\hbar = \hbar d-\phi,
\end{equation}
where $d$ is the exterior differentiation operator
acting on  sections of $E$. The form
of \eqref{hbar-connection intro} is valid because
of our choice, $(0, dx)$, as the 
first row of
the Higgs field.

For the third example of a
Gau\ss\ hypergeometric equation,
we use a particular choice of parameters
so that 
the $\hbar$-connection becomes  an $\hbar$-deformed
Gau\ss-Manin connection
 of \eqref{Gauss-Manin}.
This is a singular connection with simple poles
at $0,1,\infty$, 
and has an explicit $\hbar$-dependence in
the connection matrix. 
The Gau\ss-Manin connection $\nabla^1_{GM}$ at 
$\hbar = 1$ is equivalent to
the Picard-Fuchs equation
that characterizes the periods \eqref{elliptic periods}
of the Legendre family of elliptic curves $E(x)$
defined by the cubic equation
\begin{equation}
\label{Legendre}
t^2 = s(s-1)(s-x), \qquad x\in \cM_{0,4} = 
\bP^1\setminus \{0,1,\infty\}.
\end{equation}

The second column gives the spectral curve 
of the Higgs bundle $(E,\phi)$. 
Since the Higgs fields have 
poles, the spectral curves are no longer contained in 
the cotangent bundle $T^*\bP^1$. We need  
 the compactified cotangent bundle
\begin{equation*}
\overline{T^*\bP^1} = \bP(K_{\bP^1}\dsum
\cO_{\bP^1}) = \bF_2,
\end{equation*}
which is a Hirzebruch surface.
The parameter $y$
is the fiber coordinate of the cotangent line
$T^*_x\bP^1$. 
The first line of the second column is the
equation of the spectral curve in the $(x,y)$ 
affine coordinate of $\bF_2$.
All but the last 
example produce a singular
spectral curve. 
Let $(u,w)$ be a coordinate system
 on another affine
chart of $\bF_2$ defined by
\begin{equation}
\label{uw}
\begin{cases}
x = {1}/{u}\\
ydx = v du, \qquad w = 1/v.
\end{cases}
\end{equation}
The singularity of  $\Sigma$ 
in the $(u,w)$-plane is given by the 
second line of the second column. 
The third line of the second column gives
$\Sigma\in \NS(\bF_2)$ as an element of 
the N\'eron-Severy group of 
$\bF_2$. Here, $C_0$ is the class of the 
zero-section of $T^*\bP^1$, and $F$ represents
the fiber class of $\pi:\bF_2\lrar \bP^1$. 
We also give the arithmetic and geometric
genera of the spectral curve.

A solution $\Psi(x,\hbar)$
of \eqref{PPsi=0} for the first 
example is  given
by the \textbf{Airy function}
\begin{equation}
\label{Airy}
Ai(x,\hbar) =\frac{1}{2\pi} \hbar^{-\frac{1}{6}}
\int_{-\infty} ^\infty
\exp\left({\frac{ipx}{\hbar^{2/3}}}+{i\frac{p^3}{3}}
\right)dp,
\end{equation}
which is an entire function in $x$ for $\hbar\ne 0$.
We will perform the all-order WKB 
analysis  in this paper, and  give
a closed formula for each  term of the
WKB expansion. 
The topological recursion  produces 
the asymptotic expansion
\begin{equation}
\label{Airy expansion}
Ai(x,\hbar) =
\exp\left(\sum_{g=0}^\infty\sum_{n=1}^\infty
\frac{1}{n!}\hbar^{2g-2+n}F_{g,n}^\Airy (x)
\right)
\end{equation}
at $x=\infty$, where
\begin{equation}
\label{Airy Fgn}
F_{g,n}^{\text{Airy}}(x) 
:= \frac{(-1)^n}{2^{2g-2+n}}\cdot
x^{-\frac{(6g-6+3n)}{2}}
\sum_{\substack{d_1+\dots+d_n\\
=3g-3+n}}
\la \tau_{d_1}\cdots \tau_{d_n}\ra_{g,n}
\prod_{i=1}^n (2d_i-1)!! ,
\end{equation}
and the coefficients
$$
\la \tau_{d_1}\cdots \tau_{d_n}\ra_{g,n} =
\int_{\Mbar_{g,n}}
\psi_1 ^{d_1}\cdots \psi_{n}^{d_n}
$$
are the cotangent class intersection 
numbers on the moduli space 
$\Mbar_{g,n}$ of stable curves of 
genus $g$ with $n$ non-singular
marked points. The cases for $(g,n)=(0,1)$ and 
$(0,2)$ require a subtle care, which will be explained
in Section~\ref{sect:Airy}. The expansion 
coordinate $x^{\frac{3}{2}}$ of \eqref{Airy Fgn}
indicates the class of the irregular singularity of
the Airy differential equation.

The solutions to the second example are given by 
confluent hypergeometric functions, such as
${}_1F_1\left(\frac{1}{2\hbar};
\half;-\frac{x^2}{2\hbar}\right),
$
where 
\begin{equation}
\label{Kummer}
{}_1F_1(a;c;z) :=
\sum_{n=0}^\infty \frac{(a)_n}{(c)_n}\;
\frac{z^n}{n!}
\end{equation}
is the 
\textbf{Kummer confluent hypergeomtric function},
and the 
\textbf{Pochhammer symbol} $(a)_n$
is defined by
\begin{equation}
\label{Poch}
(a)_n :=a (a+1)(a+2)\cdots (a+n-1).
\end{equation}
For $\hbar>0$, 
the topological recursion 
determines the asymptotic 
expansion of a particular entire solution known as
 a \textbf{Tricomi confluent 
hypergeomtric function} 
\begin{multline*}
\Psi^\Catalan(x,\hbar) 
\\
= 
\left(-\frac{1}{2\hbar}\right)^{\frac{1}{2\hbar}}
\left(
\frac{\Gamma[\half]}{\Gamma[\frac{1}{2\hbar}
+\half]}
{}_1F_1\left(\frac{1}{2\hbar};
\half;-\frac{x^2}{2\hbar}\right)
+\frac{\Gamma[-\half]}{\Gamma[\frac{1}{2\hbar}]}
\sqrt{-\frac{x^2}{2\hbar}}
{}_1F_1\left(\frac{1}{2\hbar}
+\half;
\frac{3}{2};-\frac{x^2}{2\hbar}\right)
\right).
\end{multline*}
The expansion is given in the form
\begin{equation}
\begin{aligned}
\label{Catalan Psi expansion intro}
\Psi^{\Catalan}(x,\hbar)
&=
\left(\frac{1}{x}\right)^{\frac{1}{\hbar}}
\sum_{n=0}^\infty 
\frac{\hbar^n \left(\frac{1}{\hbar}\right)_{2n}}
{(2n)!!}\cdot  \frac{1}{x^{2n}}
\\
&=
\exp\left(
\sum_{g=0}^\infty 
\sum_{n=1}^\infty \frac{1}{n!}\hbar^{2g-2+n}
F_{g,n}^\Catalan(x,\dots,x)
\right).
\end{aligned}
\end{equation}
Here,
\begin{equation*}
F_{g,n}^\Catalan(x_1,\dots,x_n)
=
\sum_{\mu_1,\dots,\mu_n>0}
\frac{C_{g,n}(\mu_1,\dots,\mu_n)}
{\mu_1\cdots\mu_n}
\prod_{i=1}^n x_i^{-\mu_i}
\end{equation*}
is the generating function of the
\emph{generalized} Catalan numbers 
$C_{g,n}(\mu_1,\dots,\mu_n)$ of
\cite{DMSS, WL},
which count
the number of connected cellular graphs 
(i.e., the $1$-skeletons of cell decompositions)
of a compact surface of genus $g$ with 
$n$ labeled vertices of degrees 
$(\mu_1,\dots,\mu_n)$, together with 
an arrow attached to one of the incident 
half-edges at each vertex. For more
detail of  cellular graphs, 
we refer to \cite{DMSS,MS, WL}.
The expansion variable $x^2$ in
\eqref{Catalan Psi expansion intro} 
indicates the class of 
irregularity of the Hermite differential equation 
at $x=\infty$. The cases for $(g,n)=(0,1)$ and $(0,2)$
require again a special treatment, as we will
see later.

\begin{rem}
The authors are grateful to Peter Zograf for 
bringing
\cite{WL} to
their attention. 
The recursion for $C_{g,n}(\mu_1,\dots,\mu_n)$ 
(\cite[Theorem 3.1]{MS}, which is also equivalent
to \cite[Theorem 1.1]{DMSS}), is 
exactly the same as \cite[Equation 6]{WL}.
The topological recursion \eqref{integral TR}
for the generalized Catalan numbers derived in 
\cite[Theorem 1.2]{DMSS} is the Laplace 
transform of \cite[Equation 6]{WL}.
\end{rem}

\begin{rem}
Leonid Chekhov has shown that 
the asymptotic expansion \eqref{Catalan Psi expansion intro} can  also be derived from the matrix model
of \cite{ACNP}, by simply setting the matrix 
size equal to $1$.  The principal specialization 
often takes this effect in $1$-Hermitian matrix models.
\end{rem}

The Hermite
differential equation
 becomes simple for $\hbar=1$,
and we have the asymptotic expansion
\begin{multline}
\label{error asymptotic}
i\sqrt{\frac{\pi}{2}}e^{-\half x^2}\left[
1-\erf\left(\frac{ix}{\sqrt{2}}\right)\right]
= \sum_{n=0}^\infty \frac{(2n-1)!!}{x^{2n+1}}
\\
=\exp\left(
\sum_{2g-2+n\ge -1}\frac{1}{n!}
\sum_{\mu_1,\dots,\mu_n>0}
\frac{C_{g,n}(\mu_1,\dots,\mu_n)}
{\mu_1\cdots\mu_n}
\prod_{i=1}^n x^{-(\mu_1+\cdots +\mu_n)}
\right).
\end{multline}
Here, 
$
\erf(x) :=\frac{2}{\sqrt{\pi}} \int_0^x e^{-z^2} dz
$
is the   Gau\ss\ error function.
\begin{figure}[htb]
\centerline{\epsfig{file=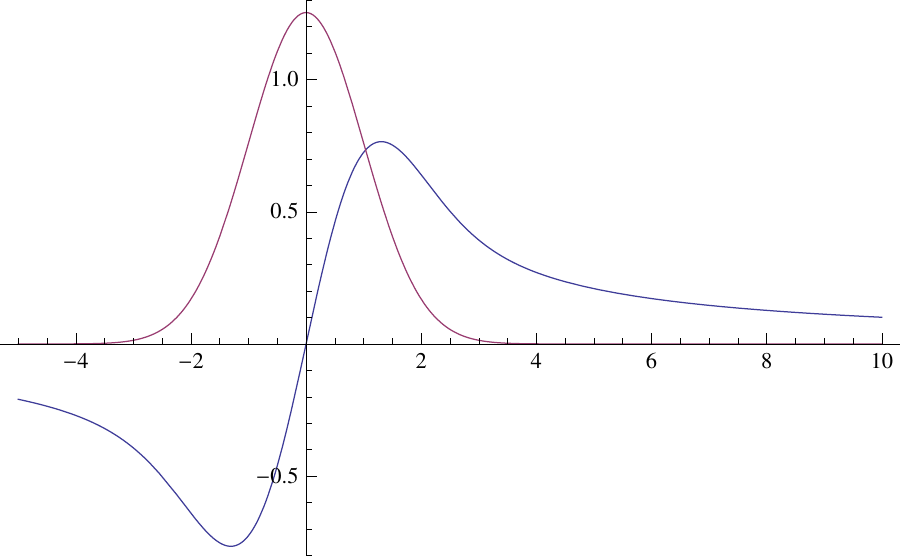, width=3in}}
\caption{The imaginary part and the real part
of $\Psi^{\Catalan}(x,1)$. For $x>>0$, the imaginary
part dies down, and only the real part has a
non-trivial asymptotic
expansion. Thus \eqref{error asymptotic} 
is a series with real coefficients. 
}
\label{fig:erf}
\end{figure}

One of the
two independent solutions to the third
example, the Gau\ss\ hypergeometric
equation, that is holomorphic
 around $x=0$ is given by 
\begin{equation}
\label{Gauss 0}
\Psi^\Gauss(x,\hbar)
= {}_2F_1
\left(-\frac{\sqrt{(h-1)(h-3)}}
   {2h}+\frac{1}{h}-
   \frac{1}{2},\frac{\sqrt{(h-1)(h-3)}}
   {2h}+\frac{1}{h}-\frac{1}{2};
   \frac{1}{h};x\right),
\end{equation}
where 
\begin{equation}
\label{Gauss}
{}_2F_1(a,b;c;x):=
\sum_{n=0}^\infty \frac{(a)_n(b)_n}{(c)_n}\;
\frac{x^n}{n!}
\end{equation}
is the \textbf{Gau\ss\ hypergeometric
function}.
The topological recursion
calculates the B-model genus
expansion of the
periods of the Legendre family of elliptic curves
\eqref{Legendre}
at the point where the 
elliptic curve degenerates to a nodal
rational curve.
For example, the procedure
applied to the spectral curve
$$
y^2
+\frac{2x-1}{x(x-1)}y+\frac{1}{4x(x-1)}=0
$$
with a choice of  
$$
\eta=\frac{-(2 x-1) - \sqrt{3x^2 - 3 x + 1}}
{2x (x-1)} dx,
$$
which is an eigenvalue $\a_1(x)$ of the Higgs
field $\phi$,
gives a genus expansion at $x=0$:
\begin{equation}
\label{Gauss expansion}
\Psi^\Gauss (x,\hbar)
=\exp\left(\sum_{g=0}^\infty
\sum_{n=1}^\infty
\frac{1}{n!}\hbar^{2g-2+n}F_{g,n}^\Gauss(x)\right).
\end{equation}
At $\hbar=1$, we have a topological recursion
expansion of the period 
$\omega_1(x)$ defined in \eqref{elliptic periods}:
\begin{equation}
\label{period expansion}
\frac{\omega_1(x)}{\pi} = \Psi^\Gauss (x,1)
=\exp\left(\sum_{g=0}^\infty
\sum_{n=1}^\infty
\frac{1}{n!}F_{g,n}^\Gauss(x)\right).
\end{equation}

A subtle point we notice here is that 
while the Gau\ss\ hypergeometric equation has
regular singular points at $x=0,1,\infty$, the
Hermite equation has an irregular singular point
of class $2$
at $\infty$. The spectral curve of each case has
an ordinary double point at $x=\infty$. But the 
crucial difference lies in the intersection of 
the spectral curve $\Sigma$ with the divisor
$C_\infty$.
For the Hermite case we have $\Sigma\cdot C_\infty
= 4$ and the intersection occurs all at once at
$x=\infty$.
For the Gau\ss\ hypergeometric case, the
intersection $\Sigma\cdot C_\infty = 4$
occurs once each at $x=0,1$, and twice at $x=\infty$. 
This \emph{confluence} of regular singular
points is the source of the irregular singularity
 in the Hermite differential equation.
 
 The fourth row indicates an example of 
 a quantum curve that has one regular singular
 point at $x=-1$ and one irregular singular point
 of class $1$ at $x = \infty$. The spectral curve
 has an ordinary double point at $x=\infty$, the same
 as the Hermite case. As 
 Figure~\ref{fig:spectral 2 and 4} shows, 
 the class of the irregular singularity at $x=\infty$
 is determined by how the spectral curve
 intersects with $C_\infty$. 
 
 \begin{figure}[htb]
\centerline{\epsfig{file=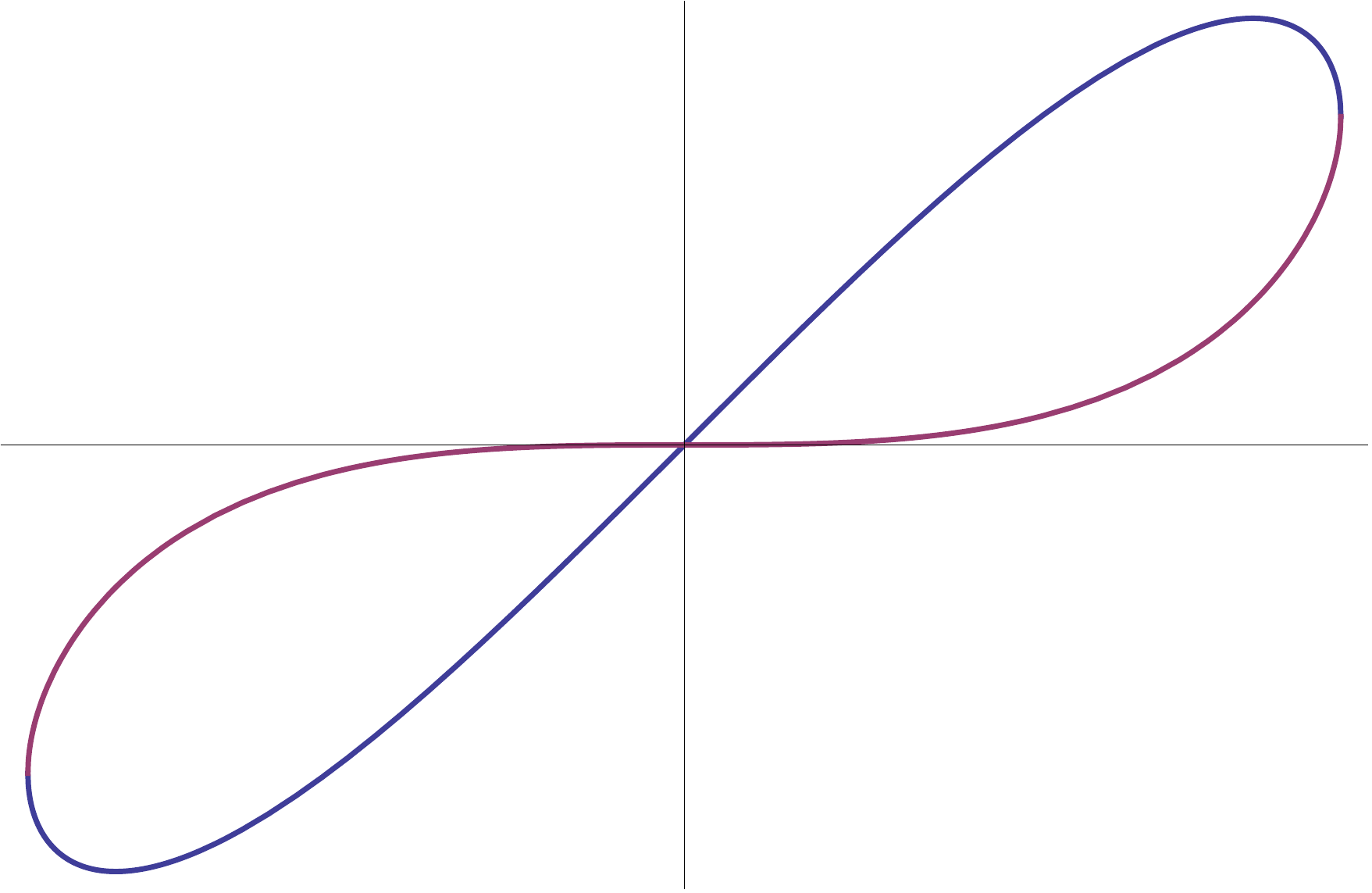, height=1.5in}\qquad
\epsfig{file=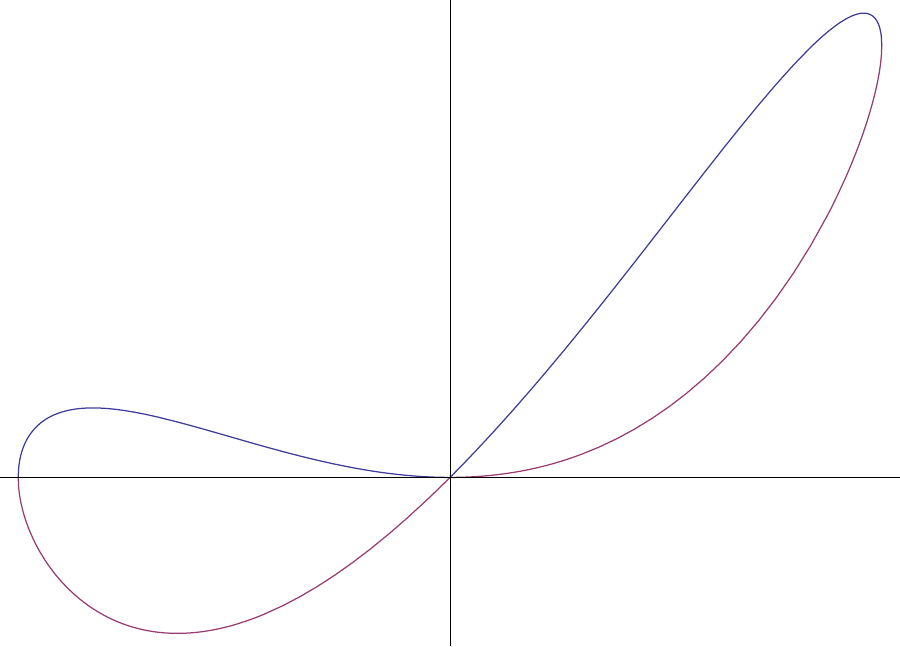, height=1.5in}}
\caption{The spectral curves  of the second
and the fourth examples.
The horizontal 
line is the divisor $C_\infty$, and the vertical
line is the fiber class $F$ at 
$x=\infty$. The spectral curve
intersects with $C_\infty$ a total of four times.
The curve on the right has a triple intersection 
at $x=\infty$, while the one on the left intersects
all at once.}
\label{fig:spectral 2 and 4}
\end{figure}

The existence of the irregular singularity in 
the quantum curve associated with a spectral
curve has nothing to do with the singularity
of the spectral curve. The fifth example shows
a non-singular spectral curve of genus $1$
(Figure~\ref{fig:spectral 5}), for which the quantum
curve has a
class $1$ irregular singularity at $x=\infty$.

\begin{figure}[htb]
\centerline{\epsfig{file=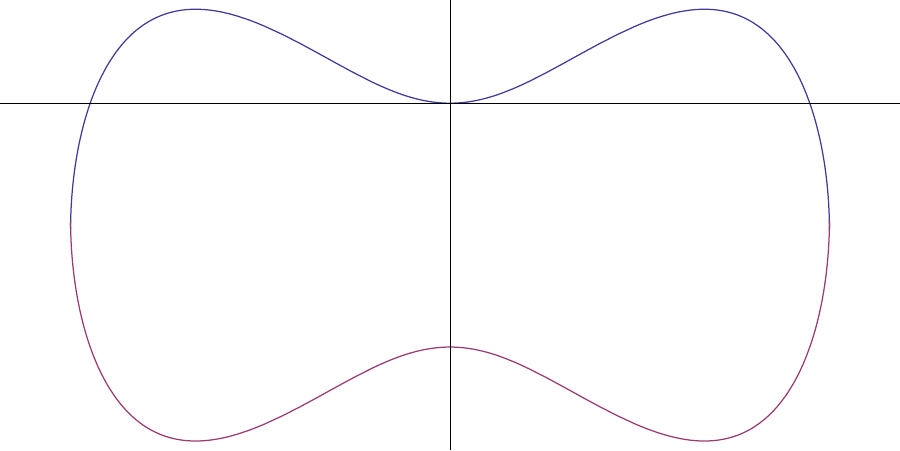, height=1.2in}}
\caption{The spectral curve  of the fifth
 example, which is non-singular. The corresponding
 quantum curve has two regular singular points
 at $x=\pm1$, and a class $1$ irregular singular 
 point at $x=\infty$.
}
\label{fig:spectral 5}
\end{figure}

 The paper is organized as follows.
 The general structure of the theory is explained 
 using the Airy function as an example in 
 Section~\ref{sect:Airy}.
 The notion of quantum curves as Rees $D$-modules
 quantizing  Hitchin spectral curves
 is presented in Section~\ref{sect:QC Rees}.
 Since our topological recursion depends
solely on the
geometry of \eqref{blow-up}, the information
of $\Sigma$ and $\widetilde{\Sigma}$, such 
as their arithmetic genera,
becomes important. We will give the genus formulas
in Sections~\ref{sect:spectral} and \ref{sect:divisors}.
 In Section~\ref{sect:spectral} we study
 the geometry of the Hitchin spectral curves
 associated with rank $2$ meromorphic Higgs
 bundles. 
 We give the genus formula
 for the normalization $\widetilde{\Sigma}$
 in terms of the characteristic polynomial of 
 the Higgs field $\phi$. 
 A more systematic treatment of the spectral curve
 and its desingularization is given in 
 Section~\ref{sect:divisors}. 
   In Section~\ref{sect:qc construction}, which is 
  the heart of our paper,
  we prove the main theorem. 
  Two more examples, Hermite differential 
  equations and Gau\ss\ hypergeometric
  differential equations, are studied in 
  Section~\ref{sect:classical}.

\section{A walk-through of the simplest example}
\label{sect:Airy}

Before going into the full generality, 
let us present 
the simplest example of our construction. With this
example we can illustrate the relation between 
a Higgs bundle, the compactified cotangent bundle
of a curve, a quantum curve, a classical differential 
equation, non-Abelian Hodge correspondence, 
and the quantum invariants that the quantum curve
captures.

As a spectral curve, we take the algebraic curve
$
\Sigma \subset \bF_2 = \bP\left(K_{\bP^1}
\dsum \cO_{\bP^1}\right) = \overline{T^*\bP^1}
$
embedded in the Hirzebruch surface 
with the defining equation
\begin{equation}
\label{Airy xy}
y^2-x = 0.
\end{equation}
Here, $x$ is the coordinate of the affine line
$\bA^1=\bP^1\setminus \{\infty\}$, and $y$ is
the fiber coordinate of the cotangent bundle
$T^*\bP^1 \subset \bF^2$ over $\bA^1$.
The Hirzebruch surface  is the
natural compactification of the cotangent bundle
$T^*\bP^1$, which is the
total space of the canonical bundle $K_{\bP^1}$.
We denote by 
$\eta \in H^0(T^*\bP^1, \pi^* K_{\bP^1})$
the tautological $1$-form associated with the 
projection $\pi:T^*\bP^1\lrar \bP^1$. It is 
expressed as $\eta = ydx$ in terms of 
the affine coordinates.
The holomorphic symplectic form 
on $T^* \bP^1$ is given by $-d\eta = dx\wedge dy$.
The $1$-form $\eta$ extends to $\bF_2$ 
as a meromorphic differential form and defines
 a divisor
\begin{equation}
\label{(eta)F2}
(\eta) = C_0-C_\infty,
\end{equation}
where $C_0$ is the zero-section of $T^*\bP^1$, and 
$C_\infty$  the section at infinity of 
$\overline{T^*\bP^1}$.
The Picard group $\Pic(\bF_2)$ of the Hirzebruch 
surface is generated by the class $C_0$ and a fiber
class $F$ of $\pi$.

Although \eqref{Airy xy} is a perfect 
parabola in the affine plane, it has a quintic cusp
singularity at $x=\infty$. 
Let $(u,w)$ be a coordinate on another affine
piece of $\bF_2$ defined by \eqref{uw}.
Then $\Sigma$ in the $(u,w)$-plane is given by
\begin{equation}
\label{Airy uw}
w^2 = u^5.
\end{equation}
The expression of $\Sigma$ as an element of 
$\Pic(\bF_2)$ is thus given by
$
\Sigma  = 2C_0 + 5 F.$
Define a stable Higgs pair
$(E,\phi)$ with
$E=\cO_{\bP^1}\dsum \cO_{\bP^1}$
and 
\begin{equation}
\label{Airy Higgs}
\phi = \begin{bmatrix}
& 1\\
x
\end{bmatrix}dx:
E\lrar E\tensor K_{\bP^1}(2) = E.
\end{equation}
Here, we choose a meromorphic 
$1$-form 
$xdx\in H^0\big(\bP^1,K_{\bP^1}
(2)\big)
$
that has a simple zero at $0\in \bP^1$ and 
a pole of order $3$ at $\infty \in \bP^1$. Up to a
constant factor,
 there is only one such  differential
$xdx = -du/u^3$.
The spectral curve $\Sigma$ 
of $(E,\phi)$ is  given by 
the characteristic equation
\begin{equation}
\label{char}
\det(\eta-\pi^*\phi) =\eta^2 -\pi^*\tr(\phi)+\pi^*
\det(\phi)= 0
\end{equation}
in $\bF_2$.
The non-Abelian Hodge correspondence
applied to $\phi$
determines a singular $\hbar$-connection
\cite{A, MS2002}
\begin{equation}
\label{NAH-Airy}
\nabla^\hbar = 
\hbar d -  \begin{bmatrix}
& 1\\
x
\end{bmatrix}dx
\end{equation}
on the trivial bundle $E=\cO_{\bP^1}^{\dsum 2}$
over $\bP^1$.

The \textbf{quantization} procedure
that we will explain in this 
paper associates the following differential equation
to the spectral curve $\Sigma$:
\begin{equation}
\label{Airy qc}
\left(\left(\hbar\frac{d}{dx}\right)^2-
x\right)
\Psi(x,\hbar) = 0.
\end{equation}
The solution $\Psi$ gives rise to a flat section
$\begin{bmatrix}\Psi\\ \Psi'\end{bmatrix}$
of \eqref{NAH-Airy}, where $'$ denotes the 
$x$ differentiation.
The differential operator
\begin{equation}
\label{Airy P}
P(x,\hbar):=
\left(\hbar\frac{d}{dx}\right)^2-
x
\end{equation}
quantizing \eqref{Airy xy} is an example
of what we call a
quantum curve.
Reflecting
the fact \eqref{Airy uw} that $\Sigma$ has a 
quintic cusp singularity at $x=\infty$, 
\eqref{Airy qc} has an 
\emph{irregular} singular point 
of \emph{class} $\frac{3}{2}$ at $x=\infty$.
This number $\frac{3}{2}$ 
indicates how the asymptotic expansion 
of the solution looks like. 
Indeed, any non-trivial solution has 
an essential singularity at $\infty$.
We note that every solution of 
\eqref{Airy qc} is an entire function
for any value of $\hbar \ne 0$. Define
\begin{equation}
\label{Airy an}
 a_{3n} =a_0\cdot
\frac{\prod_{j=1}^n (3j-2)}{(3n)!},
\qquad
 a_{3n+1} =a_1\cdot
\frac{\prod_{j=1}^n (3j-1)}{(3n+1)!},
\qquad
 a_{3n+2} = 0,
\end{equation}
for $n\ge 0$.
Then 
\begin{equation}
\label{Airy entire}
\Psi(x,\hbar):=\sum_{n=0}^\infty a_n 
\left(\frac{x}{\hbar^{{2}/{3}}}\right)^n
\end{equation}
gives an arbitrary solution to \eqref{Airy qc},
which is entire. 
The coefficients \eqref{Airy an} are of no
particular interest. 

What our quantization procedure
tells us is a different, and more interesting, story. 
Applying our main result of this paper, 
we construct a particular 
\emph{all-order} asymptotic expansion
 of this 
entire solution
\begin{equation}
\label{Psi expansion}
\Psi(x,\hbar) = \exp F(x,\hbar),\qquad 
F(x,\hbar) := \sum_{m=0}^\infty
\hbar^{m-1} S_m(x),
\end{equation}
valid for 
$|\Arg(x)|<\pi$, 
 and $\hbar >0$. Here, the first two terms of
 the asymptotic expansion are given by 
\begin{align}
\label{S0-Airy}
S_0(x) &= \pm \frac{2}{3} x^{\frac{3}{2}},
\\
\label{S1-Airy}
S_1(x) &= -\frac{1}{4} \log x.
\end{align}
Although the \emph{classical limit}
$\hbar \rar 0$
of \eqref{Airy qc} does not make  sense
under the expansion  \eqref{Psi expansion},
the \emph{semi-classical limit} through the
\emph{WKB} analysis
\begin{equation}
\label{WKB-Airy}
\left[e^{-S_1(x)}e^{-\frac{1}{\hbar}S_0(x)}
\left(\hbar^2 \frac{d^2}{dx^2}-x\right)
e^{\frac{1}{\hbar}S_0(x)}e^{S_1(x)}\right]
\exp\left(\sum_{m=2}^\infty
\hbar^{m-1} S_m(x)\right) = 0
\end{equation}
has a well-defined limit 
$\hbar \rar 0$.  The result 
is 
$
S_0'(x)^2 = x,
$
which gives 
\eqref{S0-Airy}, and also \eqref{Airy xy}
by defining $dS_0 = \eta$. This process is 
called the semi-classical limit.
The vanishing of the
$\hbar$-linear terms of \eqref{WKB-Airy}
is $2S_0'(x)S_1'(x) + S_0''(x) = 0$, which gives
\eqref{S1-Airy} above.

\begin{figure}[htb]
\centerline{\epsfig{file=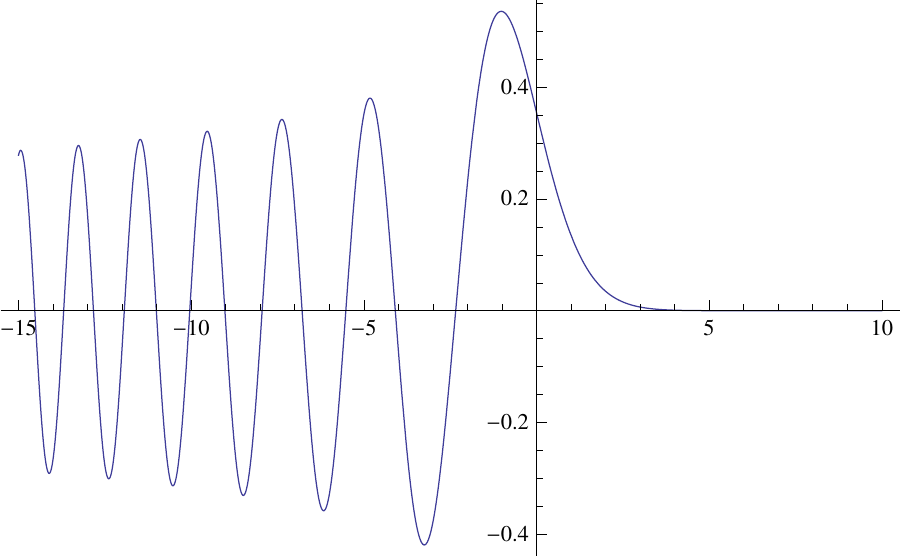, width=2in}}
\caption{The Airy function}
\label{fig:Airy}
\end{figure}

The solution we are
talking about is the Airy function \eqref{Airy}
for the choice of $S_0(x) 
= - \frac{2}{3} x^{\frac{3}{2}}$.
This solution corresponds to \eqref{Airy entire} 
with the initial condition
$$
a_0 = \frac{1}{3^{\frac{2}{3}}\Gam(\frac{2}{3})},
\qquad 
a_1 = - \frac{1}{3^{\frac{1}{3}}\Gam(\frac{1}{3})}.
$$
The surprising discovery of Kontsevich \cite{K1992}
is that $S_m(x)$ for 
$m\ge 2$ has the following \emph{closed}
formula:
\begin{equation}
\label{Airy Sm}
S_m(x) := \sum_{2g-2+n=m-1} \frac{1}{n!}
F_{g,n}^{\text{Airy}}(x),
\end{equation}
\begin{equation}
\label{Airy principal}
F_{g,n}^{\text{Airy}}(x) 
:= \frac{(-1)^n}{2^{2g-2+n}}\cdot
x^{-\frac{(6g-6+3n)}{2}}
\sum_{\substack{d_1+\dots+d_n\\
=3g-3+n}}
\la \tau_{d_1}\cdots \tau_{d_n}\ra_{g,n}
\prod_{i=1}^n (2d_i-1)!! .
\end{equation}
Although \eqref{Airy Sm} is not 
a generating function of all intersection numbers,
as we will show in the subsequent 
sections, the quantum curve
\eqref{Airy xy} alone actually determines
every intersection  number 
$\la \tau_{d_1}\cdots \tau_{d_n}\ra_{g,n}$. 
This mechanism
is the \emph{differential recursion equation} 
of \cite{DM2014}, based on the theory
of integral topological recursion of \cite{EO1}, 
which computes \emph{free energies}
\begin{equation}
\label{Airy free energy}
F_{g,n}^{\text{Airy}}(x_1,\cdots,x_n) 
:= \frac{(-1)^n}{2^{2g-2+n}}
\sum_{\substack{d_1+\dots+d_n\\
=3g-3+n}}
\la \tau_{d_1}\cdots \tau_{d_n}\ra_{g,n}
\prod_{i=1}^n \frac{(2d_i-1)!!}{\sqrt{x_i}^{2d_i+1}}
\end{equation}
as a function in $n$ variables
from $\Sigma$ through the process of 
blow-ups of $\bF_2$.

Let us now give a detailed procedure for this 
example. We start with the spectral curve $\Sigma$
of
\eqref{Airy xy}. Our goal is to 
come up with \eqref{Airy qc}.
The first step is to blow up $\bF_2$ and 
to construct \eqref{blow-up}.
The discriminant of the 
defining equation \eqref{char}
of  the spectral curve is
$$
-\det(\phi) = x(dx)^2 = \frac{1}{u^5}(du)^2.
$$
It has a simple zero at $x=0$ and a pole of order
$5$ at $x=\infty$.
The \emph{Geometric Genus Formula}
\eqref{pg} tells us that 
$\widetilde{\Sigma}$ is a non-singular curve
of genus $0$, i.e., a $\bP^1$, 
after blowing up $\lfloor \frac{5}{2}\rfloor = 2$
times.
The center of blow-up is $(u,w)=(0,0)$ for 
the first time. Put $w=w_1 u$,
and denote by $E_1$ the exceptional
divisor of the first blow-up. The proper transform
of $\Sigma$ for this blow-up,
$w_1^2 = u^3$,   has a cubic
cusp singularity, so we blow up again at
the singular point. Let $w_1=w_2 u$, 
and denote by $E_2$ the exceptional 
divisor created by the second blow-up. 
The self-intersection of the proper transform
of $E_1$ is $-2$. 
We then obtain the 
desingularized curve $\Sigma_{\min}$,
locally given by $w_2 ^2 = u$. 
The proof of 
Theorem~\ref{thm:geometric genus formula}
also tells us that $\Sigma_{\min}\lrar \bP^1$ 
is ramified at 
two points. Choose the affine coordinate 
 $t=2 w_2$ 
 of the exceptional divisor  added  at the second
 blow-up. Our  choice of the constant factor
 is to make the formula the same as in 
 \cite{DMSS}.
We have
\begin{equation}
\label{Airy xy in t}
\begin{cases}
x = \frac{1}{u} = \frac{1}{w_2^2}
 = \frac{4}{t^2}\\
y = -\frac{u^2}{w} = - \frac{u^2}{w_2u^2}
= -\frac{2}{t}.
\end{cases}
\end{equation}
In the $(u,w)$-coordinate, we see that the parameter
$t$ is  a normalization parameter of the 
quintic cusp singularity:
$$
\begin{cases}
u = \frac{t^2}{4}\\
w = \frac{t^5}{32}.
\end{cases}
$$
Note that $\Sigma_{\min}$ intersects transversally 
with the proper transform of $C_\infty$. 
We blow up once again at this intersection, and
denote by $\widetilde{\Sigma}$ the proper transform
of $\Sigma_{\min}$. The blow-up space
$Bl(\bF^2)$ is the result of $3=\lceil \frac{5}{2}
\rceil$ times blow-ups of the Hirzebruch surface.

Now we apply the differential  
recursion \eqref{differential TR}
to the geometric data
\eqref{blow-up} and \eqref{W0102}.
We claim that the integral topological recursion of
\cite{EO1} for the geometric data we 
are considering
now is exactly the same as the integral
topological 
recursion of \cite[(6.12)]{DMSS} applied to 
the curve \eqref{Airy xy in t} 
\emph{realized as a plane parabola} 
in $\bC^2$.
This is because our integral topological 
recursion \eqref{integral TR} has two 
residue contributions, one each from $t=0$ and 
$t=\infty$. As proved in \cite[Section~6]{DMSS},
the integrand on the right-hand side 
of the integral recursion formula
\cite[(6.12)]{DMSS} does not have any 
pole at $t=0$. Therefore, the residue contribution 
from this point is  $0$. The
differential recursion is obtained by 
deforming the contour of integration to enclose
only poles of the differential forms $W_{g,n}$.
Since $t=0$ is a regular point, the two methods
have no difference.

The $W_{0,2}$ of \eqref{W0102}
is simply $\frac{dt_1\cdot dt_2}{(t_1-t_2)^2}$
because $\widetilde{\Sigma}\isom \bP^1$. 
Since $t$ of \eqref{Airy xy in t}
is a normalization coordinate, 
we have
$$
W_{0,1} = \tilde{i}^*\nu^*(\eta) = y(t)dx(t)
= \frac{16}{t^4},
$$
in agreement of \cite[(6.8)]{DMSS}.
Noting that the solution to the integral
topological recursion 
is unique  from the 
initial data, we conclude that 
$$
d_1\cdots d_n F_{g,n}^\Airy \big(x(t_1),
\dots,x(t_n)\big) = W_{g,n}.
$$
By setting the constants of integration by
integrating from $t=0$ for the
differential recursion equation, we obtain
the expression \eqref{Airy free energy}.
Then its principal specialization 
gives \eqref{Airy principal}.
The equivalence of the differential 
 recursion and the quantum curve
equation Theorem~\ref{thm:main} then proves
\eqref{Airy qc} with the expression of 
\eqref{Psi expansion} and \eqref{Airy Sm}.

In this process, what is truly amazing is that
the single differential equation 
\eqref{Airy qc}, which is our 
quantum curve in this case, knows everything 
about the free energies \eqref{Airy free energy}.
This is because we can recover 
the spectral curve $\Sigma$ from the quantum 
curve. Then the procedures we need to apply, 
the blow-ups and the differential recursion equation,
are canonical. Therefore, we do recover
\eqref{Airy free energy} as explained above.

It is surprising to see that  a simple  entire function
\eqref{Airy entire}
contains so much geometric information.
Our expansion \eqref{Psi expansion} is 
an expression of an entire function viewed
from its essential singularity. We can extract 
 rich information of the solution by
restricting the region where the asymptotic
expansion is valid.
If we consider \eqref{Psi expansion} only as a
formal expression in $x$ and $\hbar$, then 
we cannot see how the coefficients are 
related to quantum invariants. 
The 
topological recursion \cite{EO1}
is a key to connect the two worlds: 
the world of quantum invariants, and the  world
of holomorphic functions and differentials. 
This relation is also knows as a
\textbf{mirror symmetry}, or in analysis,
 simply as the \emph{Laplace transform}.
The intersection  numbers 
$\la \tau_{d_1}\cdots \tau_{d_n}\ra_{g,n}$ 
belong to 
the $A$-model, while the spectral 
curve $\Sigma$ of \eqref{Airy xy}
and free energies 
belong to the $B$-model.
We consider \eqref{Airy free energy} as
an example of
the Laplace transform, playing the role of
mirror symmetry \cite{DMSS}.

\section{Quantum curves for Higgs bundles}
\label{sect:QC Rees}

In this section, we give the definition of quantum
curves. Let $C$ be a non-singular projective
algebraic curve defined over $\bC$. 
The sheaf $\cD_C$ of differential operators on $C$ is 
the subalgebra generated by the anti-canonical 
sheaf $K_C^{-1}$ and the structure sheaf
$\cO_C$ in the $\bC$-linear endomorphism algebra
$\cE nd_{\bC}(\cO_C)$. Here, $K_C^{-1}$ acts
on $\cO_C$ as holomorphic vector fields, and
$\cO_C$ acts on itself by multiplication. 
Locally every element of $\cD_C$ is written 
as 
$$
\cD_C\owns P(x) = \sum_{\ell=0}^r a_\ell(x) 
\left(\frac{d}{dx}\right)^{r-\ell}, \qquad 
a_\ell(x)\in \cO_C
$$
for some $r\ge 0$. For a fixed $r$, we introduce
the filtration by order of differential operators into
$\cD_C$ as follows:
$$
F_r\cD_C = \left.\left\{P(x) = \sum_{\ell=0}^r a_\ell(x) 
\left(\frac{d}{dx}\right)^{r-\ell}\right|
a_\ell(x)\in \cO_C\right\}.
$$
The \emph{Rees} ring $\widetilde{\cD_C}$ 
is defined by
\begin{equation}
\label{Rees}
\widetilde{\cD_C} = \bigoplus_{r=0}^\infty
\hbar^r F_r\cD_C \subset 
\bC[\hbar]\tensor_\bC \cD_C.
\end{equation}
An element of $\widetilde{\cD_C}$ on a coordinate
neighborhood $U\subset C$ can be written 
\emph{uniquely} as
\begin{equation}
\label{local Rees P}
P(x,\hbar) = \sum_{\ell=0}^r a_\ell(x) 
\left(\hbar\frac{d}{dx}\right)^{r-\ell}
\end{equation}
(see \cite[Section~1.5]{MS2002}).

\begin{Def}[Quantum curve]
\label{def:qc}
A \textbf{quantum curve} is the 
Rees $\widetilde{\cD_C}$-module
\begin{equation}
\label{Rees D-module}
\widetilde{\cM}
= \bigoplus_{r=0}^\infty
\hbar^r F_r \cM
\end{equation}
associated with 
a filtered
$\cD_C$-module $(\cM,F_r)$ 
defined on $C$, with
the compatibility
$$
F_a \cD_C\cdot F_b \cM\subset 
F_{a+b}\cM.
$$
\end{Def}

Let 
$$
D = \sum_{j=1}^n m_j p_j, \qquad m_j >0
$$ 
be an effective 
divisor on $C$.  
The point set $\{p_1,\dots,p_n\}\subset C$ is 
the support of $D$. A \emph{meromorphic
Higgs bundle} with poles at $D$ is
a pair $(E,\phi)$ consisting of an algebraic
 vector bundle
$E$ on $C$ and a Higgs field
\begin{equation}
\label{phi}
\phi:E\lrar K_C(D)\tensor_{\cO_C}E.
\end{equation}
Since the cotangent bundle 
$$
T^*C = \Spec\left(
\Sym\left(K_C^{-1}\right)\right)
$$ is the total 
space of $K_C$, we have the tautological $1$-form
$
\eta\in H^0(T^*C,\pi^*K_C)
$
on $T^*C$ coming from the  projection
$$
\begin{CD}
T^*C @<<< \pi^*K_C
\\
@V{\pi}VV
\\
C@<<< K_C.
\end{CD}
$$
The natural holomorphic symplectic form
of $T^*C$ is given by $-d\eta$. The 
\textbf{compactified cotangent bundle}
of $C$ is a
ruled surface defined by
\begin{equation}
\label{compact}
\overline{T^{*}C}:=
\bP\left(K_C\dsum \cO_C\right)
=\Proj \left(
\bigoplus_{n=0}^\infty\left(
K_C^{-n}\cdot I^0\dsum K_C^{-n+1}\cdot
I \dsum \cdots\dsum K_C^0\cdot I^n\right)
\right),
\end{equation}
where $I$ represents $1\in \cO_C$ being 
considered as a degree $1$ element. 
The divisor at infinity $C_\infty$ of 
\eqref{C-infinity}
is reduced in the
ruled surface and supported on the 
subset $\bP\left(K_C\dsum \cO_C\right) 
\setminus T^*C$. The tautological form $\eta$ 
extends on $\overline{T^{*}C}$ as a 
meromorphic $1$-form with simple poles
along $C_\infty$. Thus the divisor of $\eta$ in
$\overline{T^{*}C}$ is
given by 
\begin{equation}
\label{eta divisor}
(\eta) = C_0-C_\infty,
\end{equation}
where $C_0$ is the zero section of $T^*C$.

The relation between the sheaf $\cD_C$
and the geometry of the compactified cotangent
bundle $\overline{T^*C}$ is the following.
First we have
\begin{equation}
\label{cotangent}
 \Spec\left(
\bigoplus_{m=0}^\infty F_m\cD_C\big/
F_{m-1}\cD_C\right)
=\Spec\left(
\bigoplus_{m=0}^\infty K_C^{-m}\right)
=T^*C.
\end{equation}
Let us denote by $gr_m \cD_C = 
F_m\cD_C\big/
F_{m-1}\cD_C$.
By writing $I = 1\in H^0(C,\cD_C)$, we then 
have
\begin{multline}
\label{compactified cotangent}
 \Proj\left(
\bigoplus_{m=0}^\infty \left(gr_m\cD_C\cdot I^0
\dsum 
gr_{m-1}\cD_C \cdot I \dsum 
gr_{m-2}\cD_C \cdot I^{\tensor 2} \dsum
\cdots \dsum gr_0\cD_C \cdot I^{\tensor m}
\right)\right) 
\\
= \overline{T^*C}.
\end{multline}

\begin{Def}[Spectral curve] A \textbf{spectral curve}
of degree $r$
is a divisor $\Sigma$ in $\overline{T^*C}$ such
that the projection $\pi:\Sigma \lrar C$ defined by
the restriction
\begin{equation*}
\xymatrix{
\Sigma \ar[dr]_{\pi}\ar[r]^{i} 
&\overline{T^*C}\ar[d]^{\pi}
\\
&C		}
\end{equation*} 
is a finite morphism of degree $r$.
The \textbf{spectral curve  of a Higgs 
bundle} $(E,\phi)$ is the divisor of zeros
\begin{equation}
\label{spectral curve}
\Sigma = \left(\det(\eta - \pi^*\phi)\right)_0
\end{equation}
on $\overline{T^*C}$ of the characteristic 
polynomial $\det(\eta - \pi^*\phi)$. Here,
$$
\pi^*\phi : \pi^* E \lrar \pi^*\left(K_C(D)\right)
\tensor_{\cO_{\bP\left(K_C\dsum \cO_C\right)}}
\pi^*E.
$$
\end{Def}

\begin{rem}
The Higgs field
$\phi$ is holomorphic on 
$C\setminus \supp    (D)$. Thus we can define the
divisor of zeros
$$
\Sigma^\circ=
\left( \det\left(\eta-
\pi^*\left(\phi|_{C\setminus \supp    (D)}
\right)\right)\right)_0
$$
of the characteristic
polynomial
on $T^*(C\setminus \supp    (D))$. The spectral curve 
$\Sigma$ is the complex topology closure of 
$\Sigma^\circ$
with respect to the compactification
\begin{equation}
\label{compactification}
T^*(C\setminus \supp    (D))\subset \overline{T^*C}.
\end{equation}
\end{rem}

A left $\cD_C$-module $\cE$ on $C$ is naturally
an $\cO_C$-module with a $\bC$-linear
 integrable connection
$\nabla:\cE \lrar K_C\tensor_{\cO_C} \cE$. 
The construction goes as follows:
\begin{equation}
\label{D-module -> flat connection}
\begin{CD}
\nabla: \cE @>{\a}>> \cD_C\tensor_{\cO_C}\cE
@>{\nabla_{\!\cD}\tensor id}>>
 \left(K_C \tensor_{\cO_C}
\cD_C\right)
\tensor_{\cO_C} \cE
 @>{\b\tensor id}>> K_C\tensor_{\cO_C} \cE,
\end{CD}
\end{equation}
where
\begin{itemize}
\item $\a$ is the natural inclusion 
$\cE \owns v \longmapsto 1\tensor v\in 
\cD_C\tensor_{\cO_C}\cE$;
\item $\nabla_{\!\cD}: \cD_C \lrar K_C \tensor_{\cO_C}
\cD_C$ is the connection defined by the 
$\bC$-linear left-multiplication operation of 
$K_C^{-1}$ on $\cD_C$, which satisfies the 
derivation property
\begin{equation}
\label{D-connection}
\nabla_{\!\cD}(f\cdot P) = f\cdot 
\nabla_{\!\cD}(P)+df \cdot P
\in K_C \tensor_{\cO_C}
\cD_C
\end{equation}
for $f\in \cO_C$ and $P\in \cD_C$; and
\item $\b$ is the canonical right $\cD_C$-module 
structure in $K_C$.
\end{itemize}
If we choose a local coordinate neighborhood
$U\subset C$ with a  coordinate $x$, then
 \eqref{D-connection} takes the
 following form. Let us
denote by $P' = [d/dx, P] - P\cdot d/dx$, and 
 define
$$
\nabla_{\!\!\frac{d}{dx}}(P) :=
P\cdot \frac{d}{dx} + P'.
$$
Then we have
$$
\nabla_{\!\!\frac{d}{dx}}(f\cdot P) = f\cdot
\nabla_{\!\!\frac{d}{dx}}(P) +\frac{df}{dx}\cdot P.
$$
The connection $\nabla$ of 
\eqref{D-module -> flat connection}
is integrable because
$d^2 = 1$. Actually, the statement is true for 
any dimensions. We note that
there is no reason for $\cE$ to be
coherent as an $\cO_C$-module.

Conversely, if an algebraic vector bundle 
$E$ on $C$ of rank
$r$
admits a holomorphic  connection $\nabla:
E\lrar  K_C\tensor E$, then $E$ acquires the
structure of a  
$\cD_C$-module. This is because $\nabla$ is 
automatically flat, and
the covariant derivative $\nabla_{\!X}$
for  $X\in K_{C}^{-1}$ satisfies
\begin{equation}
\label{covariant}
\nabla_{\!X}(f v) = 
f \nabla_{\!X}( v) + X(f) v
\end{equation}
for $f\in \cO_C$ and $v\in E$. A
repeated application of \eqref{covariant} makes
$E$ a $\cD_C$-module. 
The fact that every $\cD_C$-module on a curve is
principal implies that for every point $p\in C$, 
there is an open neighborhood $p\in U\subset C$ and
a linear differential operator $P$ of
oder $r$ on $U$, called a generator, such that
$
E|_U\isom \cD_U/\cD_U P.
$
Thus on an open curve $U$,
 a holomorphic connection in a
vector bundle of rank $r$ gives rise to a differential 
operator of order $r$. The converse
is true if  $\cD_U/\cD_U P$ is $\cO_U$-coherent.

\begin{Def}[$\hbar$-connection]
\label{Def:hbar-connection}
A holomorphic $\hbar$-connection on a vector bundle
$E\lrar C$  
is a $\bC[\hbar]$-linear homomorphism
$$
\nabla^\hbar : \bC[\hbar]\tensor E
\lrar \bC[\hbar, \hbar^{-1}]\tensor
K_C \tensor_{\cO_C} E
$$ 
subject to the derivation condition
\begin{equation}
\label{hbar-connection}
\nabla^\hbar (f\cdot v) = 
f \nabla^\hbar(v) + \hbar df\tensor v,
\end{equation}
where $f\in \cO_C\tensor \bC[\hbar]$ and
$v\in \bC[\hbar]\tensor E$.
\end{Def}

\begin{rem}
The \textbf{classical limit} of a
holomorphic $\hbar$-connection
is the evaluation $\hbar = 0$ of $\nabla^\hbar$, which
is simply an $\cO_C$-module homomorphism
$$
\nabla^0: E\lrar K_C \tensor_{\cO_C} E,
$$
i.e., a holomorphic 
Higgs field in the vector bundle $E$.
\end{rem}

\begin{rem}
An 
$\cO_C\tensor\bC[\hbar]$-coherent
 $\widetilde{\cD_C}$-module
is equivalent to  a vector bundle on $C$
equipped with an 
$\hbar$-connection.
\end{rem}

In analysis, the \textbf{semi-classical limit}
of a differential operator $P(x,\hbar)$ of
\eqref{local Rees P} is defined by 
\begin{equation}
\label{P-SCL}
\lim_{\hbar\rar 0} \left(
e^{-\frac{1}{\hbar} S_0(x)}P(x,\hbar)
e^{\frac{1}{\hbar} S_0(x)}
\right) = 
\sum_{\ell=0}^r a_\ell(x)
 \left(S_0'(x)\right)^{r-\ell},
\end{equation}
where $S_0(x)\in \cO_C(U)$.
The equation
\begin{equation}
\label{SCL=0}
\lim_{\hbar\rar 0} \left(
e^{-\frac{1}{\hbar} S_0(x)}P(x,\hbar)
e^{\frac{1}{\hbar} S_0(x)}
\right) = 0
\end{equation}
then determines the first term of the 
\textbf{singular
perturbation expansion} 
\begin{equation}
\label{SPE}
\Psi(x,\hbar) = \exp\left(
\sum_{m=0} ^\infty \hbar^{m-1} S_m(x)
\right)
\end{equation}
of the solution $\Psi(x,\hbar)$
of the differential equation
$$
P(x,\hbar)\Psi(x,\hbar)=0
$$
on U.
Since $dS_0(x)$ is a local section of $T^*C$
on $U\subset C$, $y=S_0'(x)$ gives the
local trivialization   of $T^*C|_U$, with 
$y\in T_x^*C$ a fiber coordinate.
 Equations \eqref{P-SCL} and 
\eqref{SCL=0} then give an equation
\begin{equation}
\label{SCL-spectral}
\sum_{\ell=0}^r a_\ell(x) y^{r-\ell}=0
\end{equation}
of a curve in $T^*C|_U$.
This motivates us to give the following
definition:

\begin{Def}[Semi-classical limit of a 
Rees differential
operator] 
Let  $U\subset C$ be an open subset of $C$
with a local coordinate $x$ such that $T^*C$ is 
trivial over $U$ with a fiber coordinate $y$. 
The semi-classical limit of a local section 
$$
P(x,\hbar) = \sum_{\ell =0}^r a_\ell(x) 
\left(\hbar \frac{d}{dx}\right)^{r-\ell}
$$
of the 
Rees ring $\widetilde{\cD_C}$ of the 
sheaf of differential operators $\cD_C$ on $U$
is the holomorphic function
$$
\sum_{\ell =0}^r a_\ell(x) 
y^{r-\ell}
$$
defined on $T^*C|_{U}$.
\end{Def}

\begin{Def}[Semi-classical limit]
Suppose a Rees $\widetilde{\cD_C}$-module
$\widetilde{\cM}$ is written 
as
\begin{equation}
\label{D-module local}
\widetilde{\cM}(U) = \widetilde{\cD_C}(U)\big/
\widetilde{\cD_C}(U)P_U
\end{equation}
on every coordinate neighborhood $U\subset C$
with a differential operator 
$P_U$ of the form \eqref{local Rees P}.
Using this expression \eqref{local Rees P}
for $P_U$, we construct a meromorphic function
\begin{equation}
\label{local Sigma}
p_U(x,y) = \sum_{\ell=0}^r a_\ell(x) y^{r-\ell}
\end{equation}
on $\overline{T^*C}|_U$, where $y$ is the fiber
coordinate of $T^*C$, which is trivialized on $U$.
Define
\begin{equation}
\label{Sigma U}
\Sigma_U = (p_U(x,y))_0
\end{equation}
as the divisor of zero of the function $p_U(x,y)$.
If $\Sigma_U$'s glue together to a spectral 
curve $\Sigma\subset \overline{T^*C}$, then 
we call $\Sigma$ the \textbf{semi-classlical limit}
of the Rees $\widetilde{\cD_C}$-module
$\widetilde{\cM}$.
\end{Def}

\begin{rem}
For the local equation \eqref{local Sigma}
to be consistent globally on $C$, the 
 coefficients of \eqref{local Rees P}
 have to satisfy
\begin{equation}
\label{al(x) condition}
a_\ell(x)\in \Gamma\!\left(U,
K_C^{\tensor \ell}\right).
\end{equation}
\end{rem}

\begin{Def}[Quantum curve for holomorphic 
Higgs bundle]
\label{Def:QC holomorphic}
A \textbf{quantum curve} associated with 
the spectral curve $\Sigma
\subset T^*C$ 
of a  holomorphic  Higgs bundle  on 
a projective algebraic curve $C$ is a Rees
$\widetilde{\cD_C}$-module $\cE$ 
 whose semi-classical
limit is $\Sigma$.
\end{Def}

The main reason we need to extend our framework 
to   meromorphic connections is that 
there are no non-trivial holomorphic connections
on $\bP^1$, whereas many important
classical examples of differential equations
are naturally defined over $\bP^1$
with regular and irregular singularities. 
A $\bC$-linear homomorphism
$$
\nabla:E\lrar K_C(D) \tensor_{\cO_C} E
$$
is said to be a \emph{meromorphic connection} 
with poles along
an effective divisor $D$ if 
$$
\nabla(f\cdot v) = f\nabla(v) +df\tensor v
$$
for every $f\in \cO_C$ and $v\in E$. Let us denote by
$$
\cO_C(*D) :=\lim_{\lrar}\cO_C(mD),
\qquad E(*D) := E\tensor_{\cO_C} \cO_C(*D).
$$
Then $\nabla$ extends to 
$$
\nabla:E(*D)\lrar K_C(*D) 
\tensor_{\cO_C(*D)} E(*D).
$$
Since $\nabla$ is holomorphic on 
$C\setminus \supp (D)$,
it induces
a $\cD_{C\setminus \supp (D)}$-module structure 
in $E|_{C\setminus \supp (D)}$.
The $\cD_C$-module direct image
$
\widetilde{E}=j_*\left(E|_{C\setminus \supp (D)}
\right)
$
 associated with the open
inclusion map $j:C\setminus \supp (D)\lrar C$
is then naturally isomorphic to 
\begin{equation}
\label{meromorphic extension}
\widetilde{E}=j_*\left(E|_{C\setminus \supp (D)}
\right) \isom E(*D)
\end{equation}
as a $\cD_C$-module. 
\eqref{meromorphic extension} is called 
the \emph{meromorphic extension} of 
the $\cD_{C\setminus \supp (D)}$-module  
 $E|_{C\setminus \supp (D)}$.

Let us take a 
local coordinate $x$ of $C$, this time around
 a pole $p_j\in \supp(D)$. 
 If a generator $\widetilde{P}$ 
 of $\widetilde{E}$ near $x=0$ has a local expression 
\begin{equation}
\label{regular}
\widetilde{P}(x,d/dx) = 
x^k \sum_{\ell=0}^r b_\ell(x) \left(
x \frac{d}{dx}
\right)^{r-\ell}
\end{equation}
around $p_j$ with locally defined 
holomorphic functions
$b_\ell(x)$, $b_0(0)\ne 0$, and an integer $k\in \bZ$, 
then $\widetilde{P}$ has a \emph{regular} 
singular point
at $p_j$. Otherwise, $p_j$ is an \emph{irregular} 
singular point of $\widetilde{P}$.

\begin{Def}[Quantum curve for a meromorphic 
Higgs bundle]
\label{Def:QC meromorphic}
Let $(E,\phi)$ be a meromorphic 
Higgs bundle defined over a projective algebraic
curve $C$ of any genus with poles 
along an effective divisor $D$, and
$\Sigma\subset \overline{T^*C}$ 
 its spectral curve. A 
\textbf{quantum curve}
associated with $\Sigma$ is the meromorphic
extension of a Rees $\widetilde{\cD_C}$-module
$\cE$ on $C\setminus \supp (D)$ such that the
closure of its
semi-classical limit $\Sigma^\circ \subset
T^*C|_{C\setminus \supp (D)}$
in the compactified cotangent bundle 
$\overline{T^*C}$
 agrees with $\Sigma$.
\end{Def}

\section{Geometry of spectral curves in the 
compactified cotangent bundle}
\label{sect:spectral}

To construct quantum curves using 
the topological recursion, we need
a smooth  Eynard-Orantin spectral
curve 
 \cite{EO1} for which we can
 apply the recursion mechanism. When the
given Hitchin spectral curve $\Sigma$
 is singular, we have to 
find a non-singular model. 
In this paper we use the normalization 
$\widetilde{\Sigma}$ of the singular spectral 
curve.
Since the quantum curve reflects the
geometry of  $\Sigma\subset \overline{T^*C}$,
it is important
 to identify
the choice of the blow-up space 
$Bl(\overline{T^*C})$ of \eqref{blow-up}
in which $\widetilde{\Sigma}$
is realized as a smooth divisor. We then determine
the initial value $W_{0,1}$ for the topological 
recursion.

The geometry of a spectral curve also 
gives us the information of the 
singularity of the quantum curve. For example,
when we have a component of
a spectral curve tangent to the divisor $C_\infty$,
the quantum curve has an  irregular singular point,
and the class of the irregularity is determined by 
the degree of tangency. We will give a classification 
of the singularity of the quantum curves
in terms of the geometry of spectral curves
in Section~\ref{sub:quantum singularity}.

In this section, we give the construction of the
canonical blow-up space $Bl(\overline{T^*C})$,
and  determine the genus  of 
the normalization $\widetilde{\Sigma}$. 
This genus is necessary to identify the 
Riemann prime form on it, which determines
another input datum $W_{0,2}$ for the
topological recursion.

There are two different ways of defining 
the spectral curve for Higgs bundles with 
meromorphic Higgs field. Our definition 
of the previous section uses the compactified
cotangent bundle. This idea also appears in
\cite{KS2013}. The traditional definition,
which assumes the pole divisor $D$ of the
Higgs field to be reduced,
is suitable for the study of moduli spaces
of parabolic Higgs bundles.
When we deal with non-reduced effective divisors,
parabolic structures do not play any role. Non-reduced 
divisors appear naturally when we deal with 
classical equations such as the Airy differential
equation, which has an irregular singular point
of class $\frac{3}{2}$
at $\infty\in \bP^1$.

Our point of view of spectral curves is also closely
related
to considering the \emph{stable pairs} of pure
dimension $1$ on $\overline{T^*C}$. Through 
Hitchin's abelianization idea, the moduli space of 
stable pairs and the moduli space of Higgs bundles
are  identified \cite{HL}.

Let $E$ be an algebraic vector bundle 
of rank $2$ on
a non-singular projective algebraic curve
$C$ of genus $g$, and 
$$
\phi:E\lrar K_C(D)\tensor_{\cO_C} E
$$
a meromorphic Higgs field with poles along
an effective divisor $D$.
The trace and the determinant of $\phi$, 
\begin{align}
\label{a1}
a_1:=-\tr (\phi) &\in H^0\left( C,K_C(D)\right),\\
\label{a2}
a_2:=\det (\phi) &\in H^0\left( C,K_C^{\tensor 2}
 (2D)\right),
\end{align}
are well defined and determine the spectral 
curve $\Sigma$ of \eqref{spectral curve}.
For the purpose of investigating the geometry
of $\Sigma$, we do not need the information of
the Higgs bundle $(E,\phi)$, or even the pole
divisor $D$. Thus in what follows,
we only assume that $a_1$ is a meromorphic section
of $K_C$, and that $a_2$ a meromorphic 
section of $K_C^{\tensor 2}$. Then the spectral curve
is re-defined as the zero-locus in 
$\overline{T^*C}$ of a quadratic
equation with $a_1$ and $a_2$ its coefficients:
\begin{equation}
\label{spectral curve general}
\Sigma :=
\left(\eta^2+\pi^*(a_1)\eta + \pi^*(a_2)\right)_0.
\end{equation}
The only condition we impose here is that
\emph{the spectral curve is irreducible.} 
In the language of Higgs bundles, this condition
corresponds to the stability of $(E,\phi)$.

Recall that $\Pic(\overline{T^*C})$ is
generated by the zero section $C_0$ of 
$T^*C$ and fibers  of the projection map
$\pi:\overline{T^*C}\lrar C$. Since the spectral 
curve $\Sigma$ is a double covering of $C$, 
as a divisor it is  expressed as
\begin{equation}
\label{Sigma in Pic}
\Sigma = 2C_0+\sum_{j=1}^a \pi^{*}(p_j)\in 
\Pic(\overline{T^*C}),
\end{equation}
where 
$
\alpha=\sum_{j=1}^a p_j\in 
\Pic^a(C)
$
is a divisor on $C$ of degree $a$. 
As an element of the N\'eron-Severi group
$$
\NS(\overline{T^*C}) = 
\Pic(\overline{T^*C})/\Pic^0(\overline{T^*C}), 
$$
it is simply 
$$
\Sigma = 2C_0+aF \in \NS(\overline{T^*C}) 
$$ 
for a typical fiber class $F$. Since the intersection
$F C_\infty =1$, we have
$a = \Sigma C_\infty$ in $\NS(\overline{T^*C})$.
From the genus formula
$$
p_a(\Sigma) = \half \Sigma\cdot 
(\Sigma +K_{\overline{T^*C}}) +1
$$
and 
$$
K_{\overline{T^*C}} = -2C_0 + (4g-4)F 
\in \NS(\overline{T^*C}),
$$
we find that the arithmetic genus of the spectral
curve $\Sigma$ is
\begin{equation}
\label{pa}
p_a(\Sigma) = 4g-3+a,
\end{equation}
where $a$ is the number of intersections 
of $\Sigma$ and $C_\infty$.
Now we wish to find the geometric genus 
of $\Sigma$.

Motivated by the completion of square
expression of the defining equation 
\eqref{spectral curve general},
\begin{equation}
\label{square}
\eta^2+\pi^*(a_1)\eta+\pi^*(a_2)= 
\left(\eta+\half \pi^*(a_1)\right)^2 -
\left(\frac{1}{4}\pi^*(a_1)^2-\pi^*(a_2)\right)
 \end{equation}
  as a meromorphic 
 section of $\pi^*K_C^{\tensor 2}$,
 we give the following definition.

 \begin{Def}[Discriminant divisor]
The \textbf{discriminant divisor} of
the spectral curve \eqref{spectral curve general}
 is a
 divisor on $C$ defined by
\begin{equation}
\label{discriminant}
\Delta:=\left(\frac{1}{4}a_1^2 -a_2
\right) = \Delta_0-\Delta_\infty,
\end{equation}
where
\begin{align}
\label{Delta 0}
&\Delta_0= \sum_{i=1}^m m_iq_i , \quad m_i>0,
\quad q_i\in C,
\\
\label{Delta infinity}
&\Delta_\infty= \sum_{j=1}^n n_jp_j, \quad n_j>0,
\quad p_j\in C.
\end{align}
\end{Def}

Since 
$
\frac{1}{4}a_1^2 -a_2
$
is a meromorphic section of $K_C^{\tensor 2}$,
we have
\begin{equation}
\label{deg Delta}
\deg \Delta = 
\sum_{i=1}^m m_i - \sum_{j=1}^n  n_j  = 4g-4.
\end{equation}
 
 \begin{thm}[Geometric genus formula]
\label{thm:geometric genus formula}
Let us define an invariant of the discriminant 
divisor by
\begin{equation}
\label{delta}
\delta=|\{i\;|\; m_i \equiv 1 \mod 2\}|+
|\{j\;|\; n_j \equiv 1 \mod 2\}|.
\end{equation}
Then    
the geometric genus of the 
spectral curve $\Sigma$ 
of \eqref{spectral curve general}
is given by
\begin{equation}
\label{pg}
\tilde{g}:= p_g(\Sigma) = 2g-1+\half \delta.
\end{equation}
We note that   \eqref{deg Delta}
 implies  $\delta \equiv 0\mod 2$.
\end{thm}

\begin{rem}
If $\phi$ is a holomorphic Higgs field,
then $a_1 = -\tr(\phi)$, $a_2=\det(\phi)$, and
$$
m= \delta = 4g-4, \quad n=0.
$$
Therefore we have $\tilde{g} = 4g-3$, which 
agrees with the genus formula 
of \cite[Eq.(2.5)]{DM2014}.
\end{rem}

Before giving the proof of the
formula, first we wish to identify the
geometric meaning of the invariant $\delta$.
Since $\Sigma\subset 
\overline{T^*C}$ is a double covering
of $C$ in a ruled surface, locally at
every singular point  $p$,
$\Sigma$ is either irreducible, or 
reducible and consisting of two components.
When irreducible,  it is locally isomorphic
to 
\begin{equation}
\label{cusp}
t^2 - s^{2m+1} = 0, \qquad m\ge 1.
\end{equation}
If it has two components, then it is locally isomorphic
to
\be
\label{even}
t^2 -s^{2m} =(t-s^m)(t+s^m)=0.
\ee
Since the local form of
$\Sigma$ at a ramification point of
$\pi:\Sigma\lrar C$   
 is written as \eqref{cusp} with $m=0$, 
by extending the terminology  ``singularity''
to ``critical points'' of the morphism $\pi$,
we include a ramification point as a cusp
with $m=0$. 
\begin{prop}
\label{prop:geometric genus formula 2}
The invariant $\delta$ of 
\eqref{delta} counts the number
of  cusps of the spectral curve $\Sigma$. 
\end{prop}

Thus 
we have 
\begin{equation}
\label{pg 2}
p_g(\Sigma) = 2 g(C) -1+\half \;(
\text{\rm{the number
of  cusps}}).
\end{equation}

\begin{proof}[Proof of \eqref{pg 2}, assuming
Proposition~\ref{prop:geometric genus formula 2}]
Let $\nu_{\min}:\Sigma_{\min}\lrar \Sigma$ 
be the minimal 
 resolution of singularities of $\Sigma$.
 Then 
 $
 \pi_{\min}=\pi\circ\nu_{\min}:\Sigma_{\min}
 \lrar C
 $
 is a double sheeted covering of $C$ by a 
 smooth curve $\Sigma_{\min}$. If 
  $\Sigma$ has two components at a
 singularity $P$
 as in \eqref{even}, then $\pi_{\min}^{-1}(P)$
 consists of two points and $\pi_{\min}$ is 
 not ramified there. If $P$ is a cusp
 \eqref{cusp}, then 
 $\pi_{\min}^{-1}(P)$ is a ramification point 
 of the covering $\pi_{\min}$. 
 If $\delta$ counts the total number of 
 cusp singularities and the ramification points of 
 $\pi:\Sigma\lrar C$, then the 
 Riemann-Hurwitz formula
 gives us
 $$
 2-2g\left(\Sigma_{\min}\right) -\delta
 = 2\left( 2-2g(C)-\delta\right),
 $$
 which yields the genus formula \eqref{pg 2}.
\end{proof}

Since we wish to give all information of 
\eqref{blow-up} from the defining equation
\eqref{spectral curve general},
we proceed to derive a local 
structure of $\Sigma$ at each 
singularity from the global
equation in what follows. 

 \begin{proof}[Proof of Theorem~\ref{thm:geometric 
 genus formula} and 
 Proposition~\ref{prop:geometric 
 genus formula 2}]
 The proof is broken into four
 cases. Cases~3 and 4 are
  related to the Newton Polygon
 we mentioned in Introduction, \eqref{irregular class}.

 \begin{case} First we consider the case 
 with a \emph{holomorphic}
  $a_1\in H^0(C,K_C)$,
 and both  $\Delta_0$ and $\Delta_\infty$ are
 reduced. As we see below, in this case
  $\Sigma$ is non-singular, and  the two
 genera \eqref{pa} and \eqref{pg}  agree.
 \end{case}

 Let us consider the \emph{graph} 
 $\Gam_{-\half a_1}$ of the
 holomorphic $1$-form $-\half a_1$ in
 ${T^*C}$. Since $T^*C$ is the total 
 space of the canonical bundle $K_C$, the graph 
  is  a cross-section of $T^*C$. 
  We define an involution
$ \sigma:\overline{T^*C}\lrar \overline{T^*C}$
 as a reflection about $\Gam_{-\half a_1}$ along
 each fiber of $\pi$. In terms of the fiber coordinate
 $y\in T^*_xC$, it is written as
  \begin{equation}
\label{involution}
ydx\longmapsto  -ydx - a_1, \qquad ydx\in 
\overline{T^*_xC}.
\end{equation}
The spectral curve is invariant under the involution,
$\sigma:\Sigma\lrar \Sigma$, because of
\eqref{square}. By definition, 
$\Gam_{-\half a_1}\subset T^*C$ is a fixed-point set
of the involution $\sigma$. The
divisor $C_\infty$ is also fixed
by $\sigma$.
 Note that we have in this case
 \begin{equation}
 \label{Galois}
 \Gal(\Sigma/C) = \la \sigma\ra.
 \end{equation}
 Thus for a holomorphic $a_1$, the Galois
 action of $\pi:\Sigma\lrar C$ extends to the 
 whole ruled surface $\overline{T^*C}$. 
This does not hold for a meromorphic $a_1$.

 As remarked
  above, if $a_2\in H^0(C,K_C^{\tensor 2})$
is also holomorphic, then 
 $\pi:\Sigma \lrar C$ is simply branched over
$\Delta = \Delta_0$, and
$\Sigma$ is a smooth curve of genus
$4g-3$. This is in agreement of \eqref{pa}
because $n=0$ in this case.

If $a_2$ is meromorphic,  then
the pole divisor of $a_2$ is  given by
$(a_2)_\infty = \Delta_\infty$ of degree $n$.
Since $\Delta_\infty$ is reduced,
from \eqref{square} we see that 
$\pi:\Sigma \lrar C$ is ramified at 
the intersection of $C_\infty$ and 
$\pi^{-1}(\Delta_\infty)$. 
The spectral curve is also ramified at its intersection 
with $\Gam_{-\half a_1}$, which occurs
along the fibers $\pi^{-1}(\Delta_0)$. 
Note that $\deg\Delta_0 = 4g-4+n$ because 
of \eqref{deg Delta}.
Thus $\pi:\Sigma \lrar C$ is simply ramified
at a total of $4g-4+2n$ points. Therefore,
$\Sigma$ is non-singular, and we  deduce
that its genus is given by
$$
p_g(\Sigma) = p_a(\Sigma) = 4g-3+n
$$
from the Riemann-Hurwitz formula.
As a divisor class, we have
$$
\Sigma = 2C_0+\pi^*(\Delta_\infty)\in 
\Pic(\overline{T^*C}),
$$
in agreement of \eqref{Sigma in Pic}.

\begin{case}
Still $a_1\in H^0(C,K_C)$ is holomorphic, 
but $\Delta$ is non-reduced. The first example
of Table~\ref{tab:examples}, the Airy differential
equation, at $x=\infty$ falls into this category.
\end{case}

The involution \eqref{involution}
is well defined. Let $q_i\in \supp(\Delta_0)$
be a point at
 which $m_i>1$. From the global 
equation \eqref{square}, we see that the curve 
germ of $\Sigma$ at its intersection $Q$ with 
the fiber $\pi^{-1}(q_i)$ is given by a formula
$$
y^2 = x^{m_i},
$$
where $x$ is the pull-back of the base coordinate
on $C$ and $y$ a fiber coordinate, possibly tilted by
a holomorphic function in x.
We blow up once at $(x,y) = (0,0)$, using a local
parameter $y_1 = y/x$ on the exceptional
divisor. The proper transform of the curve
germ becomes
$$
y_1^2 = x^{m_i-2}.
$$
Repeat this process at $(x,y_1) = (0,0)$, until
we reach the equation
$$
y_\ell ^2 = x^{\epsilon},
$$
where $\epsilon = 0$ or $1$. The proper transform
of the curve germ is now non-singular. 
We see that after a sequence of
$\lfloor \frac{m_i}{2}\rfloor$ blow-ups starting at
the point $Q\in \Sigma \cap \pi^{-1}(q_i)$, 
the proper transform of $\Sigma$ is simply 
ramified over $q_i\in C$ if $m_i$ is odd, and
unramified if $m_i$ is even.
We apply the same sequence of blow-ups at
each $q_i$ with higher multiplicity. 

Now let $p_j\in \supp(\Delta_\infty)$ with 
$n_j>1$. The intersection $P = \Sigma\cap 
\pi^{-1}(p_j)$ lies on $C_\infty$, 
and $\Sigma$ has a singularity at $P$. Let
$z=1/y$ be a fiber coordinate of $\pi^{-1}(p_j)$
at the infinity. Then the curve germ of $\Sigma$ 
at the point $P$ is given by
$$
z^2 = x^{n_j}.
$$
The involution $\sigma$ in this coordinate is
$z\longmapsto -z$. The blow-up process we
apply at $P$ is the same as before. After 
$\lfloor \frac{n_j}{2}\rfloor$ blow-ups starting at
the point $P\in \Sigma \cap \pi^{-1}(p_j)$,
the proper transform of $\Sigma$ is simply 
ramified over $p_j\in C$ if $n_j$ is odd, and
unramified if $n_j$ is even. Again we do this
process for all $p_j$ with a higher multiplicity.

Let us  define $Bl_{\min}(\overline{T^*C})$ as
the application of a total of
\begin{equation}
\label{blow-up times}
\sum_{i=1}^m \left\lfloor \frac{m_i}{2}\right\rfloor
+\sum_{j=1}^n \left\lfloor \frac{n_j}{2}\right\rfloor
\end{equation}
times blow-ups on $\overline{T^*C}$
as described above. 
 \begin{equation}
 \label{Bl min}
\xymatrix{
\Sigma_{\min} 
\ar[dd]_{\bar{\pi}}
 \ar[rr]^{\bar{i}}\ar[dr]^{\nu_{\min}}&&Bl_{\min}(\overline{T^*C})
\ar[dr]^{\nu_{\min}}
\\
&\Sigma \ar[dl]_{\pi}\ar[rr]^{i} &&
\overline{T^*C}.  \ar[dlll]^{\pi}
\\
C 		}
\end{equation}
The proper transform $\Sigma_{\min}$ is 
the minimal desingularization of $\Sigma$.
Note that the morphism
\begin{equation}
\label{pi bar}
\bar{\pi} = \pi\circ\nu:\Sigma_{\min}
\lrar C
\end{equation}
is a double covering, ramified exactly at 
$\delta$ points. Since $p_a(\Sigma_{\min})
= p_g(\Sigma)$, \eqref{pg} follows from 
the Riemann-Hurwitz formula applied to 
$\bar{\pi}:\Sigma_{\min}\lrar C$.
It is also obvious that $\delta$ counts the
number of  cusp points of $\Sigma$, including 
smooth ramification points
of $\pi$.

\begin{case}
We are now led to considering a meromorphic 
$a_1$. Let $p\in C$ be a pole of $a_1$ of 
order $k\ge 1$.  Assume that $a_2$ also 
has a pole of oder $\ell$ at $p$, and that
$\ell > k$. The second, the third,
and the fourth  examples
 of Table~\ref{tab:examples},
 all at $x=\infty$,
 fall into this
category.
\end{case}

Choose a local coordinate $x$ of $C$ around 
$p$, and express 
$$
a_1 = \frac{c_1}{x^k},
\qquad 
a_2 = \frac{c_2}{x^\ell},
$$
where $c_1$ and $ c_2\in \cO_{C,p}$ 
are unit elements. Since both of $a_1$ and $a_2$ 
have poles at $p$, the spectral curve intersects
with $C_\infty$ along the fiber $\pi^{-1}(p)$. 
The curve germ at this intersection point
is given by the equation
\begin{equation}
\label{singular germ}
y^2 + \frac{c_1}{x^k} y + \frac{c_2}{x^\ell}
= 0,
\end{equation}
or equivalently, 
\begin{equation}
\label{regular germ}
z^2 + \frac{c_1}{c_2}x^{\ell-k} z +\frac{1}{c_2}
x^\ell = 
\left(z+\half \frac{c_1}{c_2}x^{\ell-k}\right)^2
-\frac{1}{4}\left(\frac{c_1}{c_2}x^{\ell-k}
\right)^2 + \frac{1}{c_2}
x^\ell
=0,
\end{equation}
where $z=1/y$ is a fiber coordinate of $\pi^{-1}(p)$
at infinity. Note that the coefficients of 
\eqref{regular germ} are all in $\cO_{C,p}$. 
The discriminants of \eqref{singular germ}
and \eqref{regular germ}
are given by
\begin{equation}
\label{discriminants}
\Delta_y:=
\frac{1}{4}\frac{c_1^2}{x^{2k}}-\frac{c_2}{x^\ell},
\qquad
\Delta_z:= 
\frac{1}{4}\left(\frac{c_1}{c_2}x^{\ell-k}
\right)^2 - \frac{1}{c_2}x^\ell.
\end{equation}

If $2k> \ell$, then the contribution 
from $p$ in $\Delta_\infty$ is 
$-2k p$, which does not count in $\delta$. 
Since this inequality
 is equivalent to $2(\ell-k)< \ell$,
 the contribution from $p$ in the
discriminant $\Delta_z$ is $2(\ell-k) p$. 
Locally around the  singularity, the
spectral curve is thus reducible with two
components.
We can apply the blow-up 
process of Case~2 to \eqref{regular germ}
and obtain a resolved 
curve germ unramified over $p\in C$.
For the case of
the Hermite differential equation 
given as the second example
 of Table~\ref{tab:examples}, we have
$k = 3$ and $\ell=4$.

If $2k \le \ell$, then the contribution from $p$ in
$\Delta_\infty$ is $-\ell p$. Therefore, depending on the
parity of $\ell$, it has a contribution to $\delta$.
The infinity point $x=\infty$ of the Gau\ss\
hypergeometric equation, the third
example of Table~\ref{tab:examples}, 
falls into this case, where we have
$k = 1$ and $\ell=2$.
The inequality $2k \le \ell$ is the same as 
$2(\ell-k)\ge \ell$, hence the contribution of $p$ in
$\Delta_z$ is $\ell p$. Therefore, whether
the resolved curve germ is ramified or unramified 
depends on the parity of $\ell$,
which is exactly recorded in $\delta$. If it is odd, 
then the singularity is a cusp, contributing $1$ to
$\delta$.

The above consideration shows that we need
to perform $\ell-k$ times blow-ups if $2k > \ell$,
and $\lfloor \frac{\ell}{2}\rfloor$ times 
blow-ups if $2k \le\ell$, to construct
$Bl_{\min}(\overline{T^*C})$ and 
$\Sigma_{\min}$.

\begin{case}
Finally, we assume that $a_1$ has a pole of oder
$k \ge 1$ at $p\in C$, and $a_2$ has a pole of
order $\ell$ at $p$, with $k \ge \ell$. We allow 
$a_2$ to be holomorphic at $p$. The
third example of
Table~\ref{tab:examples}, the
 Gau\ss\ hypergeomtric 
equation at $x=0,1$, and the
final example,  at $x=\pm 1,\infty$,
fall into this case.
\end{case}

The equation of the spectral curve is the 
same as \eqref{singular germ}, and its
discriminant is given by $\Delta_y$ of
\eqref{discriminants}. Since $k \ge \ell$, 
the contribution of $p$ in $\Delta_\infty$ is
$-2k p$, which is not counted in $\delta$. 
Let us re-write \eqref{regular germ}
as
\begin{equation}
\label{regular germ 2}
x^{k-\ell}z^2 +\frac{c_1}{c_2} z+ \frac{1}{c_2}
x^{k} = 0.
\end{equation}
Since $c_1/c_2 \in \cO_{C,p}$ is a unit,
we can see from this equation 
 that the curve germ passes through
$(x,z) = (0,0)$ only once as a regular point.
Indeed the discriminant of \eqref{regular germ 2}
does not vanish at $x=0$. 
In particular, $\Sigma$ is non-singular 
at its intersection of $\pi^{-1}(p)$. 
Therefore, $p\in C$ does not contribute
into the Rimann-Hurwitz formula, which
agrees with the fact that $\delta$ does not record
$p$.

This completes the proof of 
Theorem~\ref{thm:geometric genus formula},
and the fact that $\delta$ counts the total number
of odd cusps on $\Sigma$.
\end{proof}

The proof of the above theorem give us the way to
construct the blow-up space
$Bl(\overline{T^*C})$ of \eqref{blow-up}.
The data we need is not only the discriminant divisor
\eqref{discriminant}, but also 
the pole divisors of the
coefficients of the defining equation 
\eqref{spectral curve general}
of the spectral curve.
Let us write
\begin{equation}
\label{poles of a1a2}
(a_1)_\infty = \sum_{j-1}^n k_j p_j,\qquad
(a_2)_\infty = \sum_{j-1}^n \ell_j p_j,
\end{equation}
where $\{p_1,\dots,p_j\} = \supp (\Delta_\infty)$.
At each $p_j$, a Newton polygon is defined
as the upper part of the convex hull
of three points $(0,0), (1, k_j), (2, \ell_j)
\in \bR^2$, 
as in Definition~\ref{def:regular and irregular}.
We also define the invariant 
\be
\label{r}
r_j = \begin{cases}
k_j \qquad 2k_j\ge \ell_j,
\\
\frac{\ell_j}{2} \qquad 2k_j\le \ell_j.
\end{cases}
\ee

If our mission is only to resolve the singularities
of $\Sigma$, then we can use the following
blow-up method.

\begin{Def}[Construction of the minimal 
blow-up space]
\label{def:Bl min}
The minimal blow-up space $Bl_{\min}(\overline{T^*C})$ 
of \eqref{Bl min} is defined by blowing up
$\overline{T^*C}$ in the following way,
as analyzed in the proof of 
Theorem~\ref{thm:geometric genus formula}.
\begin{itemize}
\item At each $q_i$ of \eqref{Delta 0}, blow up 
at the intersection  $\Sigma\cap \pi^{-1}(q_i)$
a total of $\lfloor \frac{m_i}{2}\rfloor$ times.
\item At each $p_j$ of \eqref{Delta infinity},
  blow up 
at the intersection  $\Sigma\cap \pi^{-1}(p_j)$
a total of
\begin{enumerate}
\item $\lfloor \frac{n_j}{2}\rfloor$ times, if 
$k_j=0$, or 
 $k_j>0$ and $\ell_j \ge  2k_j$, and
\item $\ell_j - k_j$ times, if  
$k_j>0$ and  $2\ell_j > 2k_j> \ell_j$. 
\end{enumerate}
Here, $k_j$ ($\ell_j$, resp.)
is the order of pole of $a_1$ ($a_2$, resp.) at
$p_j$.
\end{itemize}
\end{Def}

\begin{rem}
The last case, $k_j>0$ and  $2\ell_j > 2k_j> \ell_j$,
is counter intuitive and does not follow the
rest of the pattern. The singularity of the spectral 
curve of the Hermite differential equation
at $x=\infty$
(the second row of Table~\ref{tab:examples})
 gives a good example. While
the pole divisor of the discriminant has 
order $6$, and the intersection of the 
spectral curve $\Sigma$
and  $C_\infty$ has degree $4$, we only need 
one time blow-up.
\end{rem}

The cumbersome definition of $Bl_{\min}(\overline{T^*C})$ becomes simple if 
we appeal to the Newton polygon. 

\begin{Def}[Construction of the  
blow-up space]
\label{def:Bl}
The  blow-up space $Bl(\overline{T^*C})$ 
of \eqref{blow-up}
is defined by blowing up
$\overline{T^*C}$ in the following way.
\begin{itemize}
\item At each $q_i$ of \eqref{Delta 0}, blow up 
at the intersection  $\Sigma\cap \pi^{-1}(q_i)$
a total of $\lfloor \frac{m_i}{2}\rfloor$ times.
\item At each $p_j$ of \eqref{Delta infinity},
  blow up 
at the intersection  $\Sigma\cap \pi^{-1}(p_j)$
a total of
$\lceil r_j\rceil$ times.
\end{itemize}
\end{Def}

\begin{thm}
\label{thm:nonsingular}
In the blow-up space 
$Bl(\overline{T^*C})$, we have the following.
\begin{itemize}
\item
The proper transform 
$\widetilde{\Sigma}$
of the spectral curve
$\Sigma \subset \overline{T^*C}$ 
by the birational morphism
$\nu : Bl(\overline{T^*C})\lrar
\overline{T^*C}$ is a smooth curve with 
a holomorphic map
$\tilde{\pi}=\pi\circ\nu:\widetilde{\Sigma} \lrar C$.

\item The proper transform of $C_\infty$ and
$\widetilde{\Sigma}$ do not intersect in 
$Bl(\overline{T^*C})$. 

\item 
The Galois action
 $\sigma:\Sigma\lrar \Sigma$
lifts to an involution of $\widetilde{\Sigma}$,
and the morphism 
\begin{equation}
\label{pi-tilde}
\tilde{\pi}:\widetilde{\Sigma}
\lrar C
\end{equation}
is a Galois covering with the Galois group 
$\Gal(\widetilde{\Sigma}/C) =\la \tilde{\sigma}
\ra  \isom \bZ/2\bZ$.
\begin{equation}
\label{sigma-tilde}
\begin{CD}
\widetilde{\Sigma}@>{\nu}>>\Sigma @>{\pi}>> C
\\
@V{\tilde{\sigma}}VV @VV{\sigma}V @|
\\
\widetilde{\Sigma}@>>{\nu}> \Sigma
@>>{\pi}> C
\end{CD}
\end{equation}
\end{itemize}
\end{thm}

\begin{proof}
The curve we are trying to desingularize is the
support $\Sigma \cup C_\infty$ of the total divisor
\eqref{total} of the characteristic polynomial. 
Even $\Sigma$ is smooth at its intersection
with $C_\infty$, the support $\Sigma \cup C_\infty$
is always
singular. 
The key point is that the $\lceil r_j \rceil$ times
blow-up at the intersection is exactly what we need
to desingularize  
$\Sigma\cup C_\infty$.

Let $P\in \Sigma\cap C_\infty$ so that 
$\pi (P) = p_j$. We drop the index $j$ in the
rest of this proof.

The only case $P$ is a smooth point of 
$\Sigma$ is Case~4. From \eqref{regular germ 2},
we see that the spectral curve near $P$ is
given by $z=x^k$. Thus it is tangent to 
$C_\infty$ with the multiplicity $k$. 
Therefore, we need $k$ times
 blow-ups to separate the proper transforms
 of $\Sigma$ and $C_\infty$. 
By \eqref{r},
we have $r=k$.

The point $P$ is a cusp singularity only 
when $0\le 2k\le \ell$ and $\ell$ is odd.
We have $r = \ell/2$, and we need
$\lfloor \ell/2\rfloor$-times blow-ups to 
desingularize $\Sigma$ at $P$. To separate
the proper transforms of $\Sigma$ and $C_\infty$
in the end, we need one more blow-up. 
Therefore, we need a total of $\lceil r \rceil$ blow-ups.

If $0\le 2k\le \ell$ and $\ell$ is even, then 
still we have $r = \ell/2$. In this case the
spectral curve is locally reducible at $P$, 
and requires $r$-times blow-ups for desingularization.
Since the proper transform of $P$ consists of
two distinct points, the proper transforms 
of $\Sigma$ and $C_\infty$ are separated
after $r$-th blow-up.

The remaining case is $0< k <\ell < 2k$.
We have $r=k$. We need only $(\ell-k)$ times
blow-ups to desingularize the 
spectral curve at $P$. Let us take a close
look at \eqref{regular germ}. We take
$z^2=0$ to see the infinitesimal relation between
$\Sigma$ and $C_\infty$. Then the equation
becomes
\be
\label{tangent component}
c_1 z +x^k = 0,
\ee
which represents an irreducible component of 
the spectral curve near $P$ that is tangent to 
$C_\infty$. The degree of tangency is
$k$, hence it requires $k$-times blowing up to 
separate the proper transforms of $\Sigma$
and $C_\infty$.

Since the spectral curve $\Sigma$ is a double covering
of $C$, $\Gal(\Sigma/C) \isom \bZ/2\bZ$. The
involution $\sigma$ is its generator, which may or
may not extend to the whole $\overline{T^*C}$. 
Since we construct $\Sigma_{\min}\subset
Bl(\overline{T^*C})$ as a simply ramified
double covering over $C$ in 
Theorem~\ref{thm:geometric genus formula},
it is non-singular and there is a natural involution
on it. The additional blow-ups 
$$
\bar{\nu}:Bl(\overline{T^*C})\lrar
Bl_{\min}(\overline{T^*C})
$$
does not affect the proper transform 
$\widetilde{\Sigma}$ of $\Sigma_{\min}$, 
which also has  an involution $\tilde{\sigma}$.
The involution 
$\tilde{\sigma}$ agrees with $\sigma$ on the 
complement of the singular locus of $\Sigma$, 
thus satisfying $\nu\circ \sigma = \tilde{\sigma}
\circ \nu$.

This completes the proof.
\end{proof}

\section{The spectral curve as a divisor and its
minimal resolution}
\label{sect:divisors}

In this section we give a formula for the
minimal resolution $\Sigma_{\min}$
of the spectral curve $\Sigma$ as an 
element of the Picard group
$\Pic\big(Bl_{\min}(\overline{T^*C})\big)$.
We also give a genus formula for 
$\Sigma_{\min}$  in terms of its geometry 
in $Bl_{\min}(\overline{T^*C})$.
This gives another interpretation of the
invariant $\delta$ of
the genus formula \eqref{pg}.

The Picard group of $\overline{T^*C}$ is generated by the pull-back of $\Pic(C)$ and the zero-section $C_0$ 
of the cotangent bundle $T^*C$.
Denote by $F$ the fiber over a point of the morphism 
$\pi:\overline{T^*C} \lrar C$. We have the following intersection table:

\begin{equation}\label{intersection table 1}
\begin{split}
&F^2=0\\
&C_0^2=2g(C)-2\\
&FC_0=1.
\end{split}
\end{equation}
For the sake of simplicity, in what follows
 we denote  simply
\be
\label{convention}
\a F:=
\pi^*(\a)
\ee
 for any divisor $\a\in \Pic (C)$ (see \cite[Chapter V.2]{hartshorne}).
In this notation, the canonical divisor of 
$\overline{T^*C}$ is given by
\be
\label{canonical ruled}
K_{\overline{T^*C}} = -2C_0 + \beta F
\ee
with a divisor $\beta\in \Pic^{4g-4} (C)$ of degree $4g-4$.

In Section~\ref{sect:spectral} we identified a
 spectral curve $\Sigma$ as a divisor in the ruled surface $\overline{T^*C}$. With the convention of
\eqref{convention}, \eqref{Sigma in Pic} reads
\begin{equation}
\label{Sigma in alpha}
\Sigma=\alpha F+2C_0,
\end{equation}
where 
$$
\alpha = \Sigma C_\infty
\in \Pic (C_\infty) \isom \Pic(C)
$$ 
is a divisor of degree $a$ on $C$. 

Since we   deal only with the spectral curve of 
a rank $2$  Higgs bundle, its singularities are mild,
and  we can describe their resolution in detail.   A singular spectral curve has \emph{infinitely near} singularities, which require iterative sequence of blow-ups on the ruled surface to be resolved.  
For every singular point, say $P\in \Sigma$,
we introduce a sequence of blow-ups
\begin{equation}
\label{blow up P}
Bl_{i}^P(\overline{T^*C})\stackrel{\nu_{i}^P}{\longrightarrow} 
\cdots \stackrel{\nu_{2}^P}{\longrightarrow}Bl_1^P(\overline{T^*C})
\stackrel{\nu_{1}^P}{\longrightarrow}
Bl_0^P(\overline{T^*C})=\overline{T^*C}.
\end{equation}
Here, 
$
\nu_{j+1}^P: Bl_{j+1}^P(\overline{T^*C})\lrar
Bl_{j}^P(\overline{T^*C})$,
$j=1,\cdots,i-1$, is a blow-up
at the singular point
of $\Sigma_{j}
\subset Bl_{j}^P(\overline{T^*C)}$,
and $\Sigma_{j}$ is the proper transform
of $\Sigma_{j-1}\subset Bl_{j-1}^P(\overline{T^*C)}$
under the blow-up $\nu_j^P$. 
Each $\nu_j^P$ introduces an exceptional divisor $E_j^P$ with self-intersection $-1$ on $Bl_j^P(\overline{T^*C}).$ 
By abuse of notation, we also write
\be
\label{EjP}
E_j^P:={(\nu_{j+1}^P)}^*(E_j^P)-E_{j+1}^P, 
\ee
which is 
the proper transform of the divisor $E_j^P$ 
on $Bl_{j}^P(\overline{T^*C})$ by $\nu_{j+1}^P$.
On  $Bl_{i+1}^P(\overline{T^*C})$, we  have 
a chain of self-intersection
$-2$ curves with the following intersections
properties:
\begin{equation}
\label{intersection table 2}
\begin{split}
&{(E_1^P)}^2=\ldots={(E_i^P)}^2=-2\\
&{(E_{i+1}^P)}^2=-1\\
&E_{j-1}^PE_{j}^P=1, \textrm{ for all } 2\leq j\leq i+1\\
&E_j^PE_k^P=0, \textrm{ for all } 1\leq j, k \leq i+1 \textrm{ with } |j-k|> 1 .
\end{split}
\end{equation}
From
\eqref{Sigma in alpha}, we see that $\Sigma$
has  only infinitely near double singularities. 
We denote by
\begin{equation}\label{blow up 2P}
\Sigma_{i}\stackrel{\nu_{i}^P}{\longrightarrow} 
\cdots \stackrel{\nu_{2}^P}{\longrightarrow}\Sigma_1
\stackrel{\nu_{1}^P}{\longrightarrow}\Sigma_0=
\Sigma
\end{equation}
the sequence of proper transforms 
of $\Sigma$ under
\eqref{blow up P}. The \textbf{multiplicity} of 
the singularity $P\in \Sigma$ is defined to be
the least $i$ of \eqref{blow up 2P} such 
that $\Sigma_i$ is non-singular.

\begin{thm}
\label{thm:Sigma min}
Let $Q_1,Q_2,\dots,Q_c$ be the singularities of
$\Sigma$ not on the zero section $C_0$ or
on the divisor $C_\infty$, and 
$Q_{c+1},\dots,Q_s$ the singularities on
$C_\infty$. We denote by $n_k$  the 
multiplicity of $Q_k$, $k=1,2,\dots,s$.  
Let 
\be
\label{minimal resolution}
\begin{CD}
\Sigma_{\min}\subset
Bl_{\min}(\overline{T^*C})
@>{\nu_{\min}}>>{T^*C}
\end{CD}
\ee
 be the minimal
resolution of $\Sigma$ after performing the 
blow-ups at each singularity exactly as required,
which includes blow-ups on the singularities
on $C_0$. Then the genus of the smooth 
curve $\Sigma_{\min}$ is given by
\be
\label{Sigma min genus}
g(\Sigma_{\min}) = 2g-1 + \frac{N_0+N_\infty}{2}
-\sum_{k=1}^c n_k,
\ee
where $N_0$ (resp.\ $N_\infty$) is the number
of intersection points of $\Sigma_{\min}$
with the proper transform of 
$C_0$ (resp.\ $C_\infty$) in 
$Bl_{\min}(\overline{T^*C})$.

Denote by $E_j ^k$
the exceptional divisor of \eqref{EjP} for
$P=Q_k$, $k=1,2,\dots,s$. Then  as an
element of the  Picard group,
we have
\be
\label{Sigma min}
\Sigma_{\min} = 2C_0 + \a F -
2\sum_{k=1}^s  \sum_{j=1}^{n_k} j E_j^k
\in \Pic\left(Bl_{\min}(\overline{T^*C})
\right),
\ee
where $\a\in \Pic(C)$ is a divisor on $C$ of degree
\be
\label{a in min resolution}
a= N_\infty + 2\sum_{k=c+1}^s n_k.
\ee
\end{thm}

\begin{rem}
If $\Sigma$ is smooth, then $N_0-N_\infty = 4g-4$
and $a=N_\infty$. Therefore, 
\eqref{Sigma min genus} agrees with 
\eqref{pa}.
\end{rem}

\begin{proof}
Let us  denote by $P_1,..., P_t$ the 
 singular points of $\Sigma$ on the zero section $C_0$
with multiplicities $m_1,\dots,m_t$.  To avoid 
confusion, we denote by $G_{j}^k :=E_j^{P_k}$ the exceptional divisor
of \eqref{EjP} at $P=P_k$. The construction
of $\Sigma_{\min}$ and 
$Bl_{\min}(\overline{T^*C})$
requires also sequences of blow-ups at these points.
We use the same notation $C_0$ for  the proper
transform of the zero section via any of the 
blow-up appearing in this construction of
the minimal resolution. 

The Picard group 
$\Pic\left(Bl_{\min}(\overline{T^*C})
\right)$  is generated by $\Pic(C)$, the divisors $E_{j}^k$'s and $G_{j}^k$'s, and $C_0$. These generators   satisfy, in addition to \eqref{intersection table 2},  the following:
\begin{equation}\label{intersection table 3}
\begin{aligned}
&C_{0}^2=2g-2-\sum_{k=1}^{t}m_k\\
&C_{0}G_{m_k}=1, \text{ for every $1\leq k\leq t.$ }
\end{aligned}
\end{equation}
Since the singular points of the spectral curve are not in general position, to give an explicit expression for $\Sigma_{\min}$ as a divisor,
 we consider two separate cases.
\smallskip

\noindent
{(1)} Resolving singularities of $\Sigma$ on the zero section $C_0$.
\smallskip

\noindent
For a singular point $P_k$, the resolution 
$\Sigma_{m_k}$ is of the form
\begin{equation}
\label{Sigma k}
\Sigma_{m_k}=\alpha_k F+2C_{0},
\end{equation}
where $\alpha_k \in \Pic(C)$.
At each step of the blow-ups, the canonical divisor of 
$Bl_{j}^{P_k}(\overline{T^*C})$, for 
$j\le m_k$, is $K_j=(\nu_{j})^*(K_{j-1})+G_{j}$.
Therefore,
\begin{align*}
K_j&=(\nu_{j})^*(K_{j-1})+G_{j}\\
&=-2(C_0+G_j)-(\nu_{j})^*\left(\sum_{\ell=1}^{j-1}\ell G_\ell\right)+\beta F+G_j\\
&=-2C_0-\sum_{\ell =1}^{j-2}\ell G_\ell-(j-1)(G_{j-1}+G_{j})-G_{j}+\beta F\\
&=-2C_0+\beta F-\sum_{\ell=1}^{j}\ell G_\ell,
\end{align*}
where $\beta$ is the divisor 
of \eqref{canonical ruled}. 

\smallskip

\noindent 
(2) Resolving singularities of $\Sigma$ not on the zero section.
\smallskip

\noindent
We now consider the singular point $Q_k$.  The proper transform of $\Sigma$ under $n_k$ iterated blow-ups is
\begin{equation}
\label{spectral 2}
\Sigma_{n_{k}}=\Sigma- 2\sum_{i=1}^{n_k} \left(\sum_{j=1}^{i}E_{n_k-j+1}\right)
=\Sigma - 2\sum_{i=1}^{n_k}iE_{i}.
\end{equation}
The canonical divisor on the blown up ruled surface at the point $Bl_{i}^{Q_k}(\overline{T^*C})$ is 
\begin{align*}
K_i&=(\nu_{i})^*(K_{i-1})+E_{i}\\
&=-2C_0+\beta F+(\nu_{i})^*
\left(\sum_{j=1}^{i-1}jE_j\right)+E_i\\
&=-2C_0+\beta F+\sum_{j=1}^{i-2}jE_j+(i-1)(E_{i-1}+E_{i})+E_{i}\\
&=-2C_0+\beta F+\sum_{j=1}^{i}jE_j.
\end{align*}
Alter all these blowups, we obtain
the expression of the canonical divisor on $Bl_{\min}(\overline{T^*C})$:
\begin{equation}\label{canonical}
K_{Bl_{\min}(\overline{T^*C})}=-2C_0+\beta F+\sum_{k=1}^{s} \sum_{i=1}^{n_k}iE_{i}^{k}-\sum_{k=1}^{t}\sum_{i=1}^{m_k}iG_{i}^{k}.
\end{equation}
From \eqref{Sigma k} and \eqref{spectral 2} we
obtain 
\begin{equation}
\label{spectral min}
\Sigma_{\min}=(2C_0+\alpha F)-2\sum_{k=1}^{s}\sum_{i=1}^{n_k}iE_{i}^{k},
\end{equation}
where $\a$ is the sum of all $\a_k$ of
\eqref{Sigma k}.

Let us now turn our attention to 
determining the degree of $\a$.
We recall the genus formula
$$p_a(\Sigma_{\min})=\frac{\Sigma_{\min}\left(\Sigma_{\min}+K_{Bl_{\min}(\overline{T^*C})}\right)}{2}+1.$$
Equations
 \eqref{canonical} and \eqref{spectral min} yield
$$\Sigma_{\min}+K_{Bl_{\min}(\overline{T^*C})}=(\alpha+\beta)F - 
\sum_{k=1}^{s} \sum_{i=1}^{n_k}iE_{i}^{k}-
\sum_{k=1}^{t}\sum_{i=1}^{m_k}iG_{i}^{k}.
$$
Denoting $\deg \a = a$, the above gnus formula yields
\begin{align*}
2p_a(\Sigma_{\min})&-2
\\
&=2(a+4g-4)C_0F
-2\sum_{k=1}^{t}\sum_{i=1}^{m_k}iC_0G_{i}^{k}
+2\left(\sum_{k=1}^{s}
 \sum_{i=1}^{n_k}iE_{i}^{k}\right)^{2}
\\
&=2(a+4g-4)
-2\sum_{k=1}^{t}m_kC_0G_{m_k}^{k}
+2\sum_{k=1}^{s}
\left(\sum_{i=1}^{n_k-1}(-2)i^2-n_k^2+2\sum_{i=1}^{n_k-1}i(i+1)\right)
\\
&=2(a+4g-4)-2\sum_{k=1}^{t}m_k
+2\sum_{k=1}^{s}
\left(-n_k^2+2\sum_{k=1}^{n_k-1}i\right)
\\
&=2(a+4g-4)-2\sum_{k=1}^{t}m_k
-2\sum_{k=1}^{s}n_k.
\end{align*}
We therefore conclude that
\begin{equation}
\label{genus in m n}
g(\Sigma_{\min})=a+4g-3-\sum_{k=1}^{t}m_k
-\sum_{k=1}^{s}n_k.
\end{equation}

Since the proper transform of
$C_0$ on $Bl_{\min}(\overline{T^*C})$ does not
intersect with the exceptional divisors 
$E_{j}^{k}$'s, from  \eqref{spectral min} we have
$$
\Sigma_{\min}\cdot C_{0}=\left(\alpha F+2C_{0}
\right)C_{0}=N_{0}.
$$
This yields
\be
\label{a N0}
4g-4+a=\sum_{i=k}^{t} 2m_{k}+N_{0}.
\ee
The proper transform of $C_\infty$
on $Bl_{\min}(\overline{T^*C})$ is 
given by
\be
\label{pt C infinity}
C_{\infty}-\sum_{k=1}^{t}
\sum_{j=1}^{m_k}jE_{j}^{k},
\ee
which we also denote simply by $C_\infty$ if
there is no confusion. 
We recall that on $Bl_{\min}(\overline{T^*C})$
we have 
\begin{align*}
&{C_{\infty}}F=1,\qquad
{C_{\infty}}C_{0}=0,
\\
&C_\infty E_j^{k} = \begin{cases}
n_k \qquad j=n_k, k=c+1,\dots,s\\
0 \qquad \;\;\text{otherwise.}
\end{cases}
\end{align*}
Thus from the intersection of 
\eqref{spectral min} and $C_\infty$, we obtain
\be
\label{a proved}
a= N_\infty + 2\sum_{k=c+1}^s n_k.
\ee
This proves \eqref{a in min resolution}.

Substituting \eqref{a proved} in \eqref{a N0}, 
we obtain
$$
4g-4=2\sum_{k=1}^{t} m_{k}
-2\sum_{k=c+1}^{s} n_{k}+N_{0}-N_{\infty},
$$
which shows that $N_0\pm N_\infty$ is even. 
From \eqref{genus in m n} we have
\begin{align*}
g(\Sigma_{\min})&=(a+4g-3)-\sum_{k=1}^{t}m_k
-\sum_{k=1}^{c}n_k- \sum_{k=c+1}^{s}n_k
\\
&=1+N_0+\sum_{k=1}^{t}m_k
-\sum_{k=1}^{c}n_k- \sum_{k=c+1}^{s}n_k.
\end{align*}
From the above two equations, we obtain
$$
4g-4=2g(\Sigma_{\min})-2
+2\sum_{k=1}^{c}n_k-N_0-N_\infty,
$$
which yields
$$
g(\Sigma_{\min}) = 2g-1 + \frac{N_0+N_\infty}{2}
-\sum_{k=1}^c n_k.
$$
This completes the proof of 
Theorem~\ref{thm:Sigma min}.
\end{proof}

\section{Construction of the quantum curve}
\label{sect:qc construction}

When we say that the quantization of a characteristic
equation 
\be
\label{algebraic}
 \eta^2 -\pi^*\tr(\phi) \eta + \pi^*\det(\phi) = 0
\ee
of a Higgs field $\phi$ is a differential Žquation
\be
\label{quantum equation}
\left(\left(\hbar\frac{d}{dx}\right)^2
-\tr(\phi(x))\left (\hbar\frac{d}{dx}\right) 
+ \det(\phi(x))\right)\Psi(x,\hbar) = 0,
\ee
it may sound obvious. The point is that since $x$
and the differential operator $d/dx$ do not commute,
there are many different differential equations
other than \eqref{quantum equation}
that correspond to the starting algebraic equation
\eqref{algebraic}.
The mechanism we use to identify the 
correct formula for the quantization is 
the topological recursion. 
In this section, 
first we formulate our main theorem of 
quantization. Then 
we give the definition of the topological 
recursion, using the desingularization of the 
spectral curve \eqref{blow-up}, 
constructed in Theorem~\ref{thm:nonsingular} and Definition~\ref{def:Bl}.
The rest of the section is devoted to proving 
the main theorem.

\subsection{The main theorem}
\label{sub:main}

The main theorem of this paper is the construction 
of the quantum curve guided by the asymptotic 
expansion of its solutions, which is
obtained by the topological
recursion.

\begin{thm}[Main Theorem]
\label{thm:main}
Let $C$ be a smooth projective curve of an
arbitrary genus, and $(E,\phi)$ a rank $2$ Higgs
bundle consisting of a topologically
trivial vector bundle $E$ and   an arbitrary
meromorphic Higgs field $\phi$.
We denote by $\Sigma\subset \overline{T^*C}$
 the spectral
curve defined by \eqref{spectral curve}.
Then there exists a
Rees $D$-module $\widetilde{\cM}$
on $C$ whose semi-classical 
limit agrees with the spectral curve
$\Sigma $. On every coordinate chart $U$ of
$C$ with a local coordinate $x$, 
a generator of $\widetilde{\cM}$
is given by a differential operator
\be
\label{qc main}
P(x,\hbar) = \left(\hbar\frac{d}{dx}\right)^2
-\tr\phi(x) \left(\hbar\frac{d}{dx}\right) +\det \phi(x)
\in \widetilde{D_C}(U),
\ee
so that we have
\be
\label{local m}
\widetilde{\cM} = 
\widetilde{D_C}(U)\big/
\widetilde{D_C}(U)\cdot P(x,\hbar).
\ee

Let $q\in C$ be one of the 
 critical 
values of the projection $\pi:\Sigma \lrar C$
that corresponds to a branch point of
the desingularized covering 
$\tilde{\pi}:\widetilde{\Sigma}\lrar C$. Then
 there exists a coordinate  neighborhood 
$U_q\subset C$ of $q$ with a coordinate $x$ centered 
at $q$ such that the following holds.
\begin{enumerate}
\item
For an arbitrary point $p\in U_q$,
there is a contractible open neighborhood 
$V_p\subset U_q$ of $p$ that does not 
contain $q$.

\item 
Choose an eigenvalue $\a$ of $\phi$ on $V_p$. 
Then there is 
an all-order asymptotic 
solution to the differential equation
\be
\label{PPsi alpha}
P(x,\hbar)\Psi^\a (x,\hbar) = 0
\ee
that is defined on $V_p$.
\item The asymptotic expansion is given by
\be
\label{WKB alpha}
\Psi^\a(x,\hbar) = 
\exp\left(\sum_{m=0}^\infty \hbar^{m-1}
S_m^\a (x)
\right).
\ee
Here, 
\begin{itemize}
\item the $0$-th term $S_0^\a(x)$  is determined
by solving \eqref{scl};
\item the first term $S_1^\a(x)$ is determined by
solving \eqref{consistency}; 
\item $S_m^\a(x)$ for $m\ge 2$ is given by
\be
\label{S in F alpha}
S_m^\a(x) = \sum_{2g-2+n=m-1} 
\frac{1}{n!} F_{g,n}^\a(x); 
\ee
\item  the free energies 
$F_{g,n}(z_1,\cdots,z_n)$ for $2g-2+n>1$
are determined 
by the differential recursion
\eqref{differential TR};
\item  and each $F_{g,n}^\a(x)$ is the 
\textbf{principal
specialization} of the restriction of
the  free energy 
 to the 
open subset $V_\a\subset  \Sigma$
of $\Sigma$ that corresponds to the eigenvalue 
$\a$ on $V_p$, which we identify with $V_p$ 
by $\pi:V_\a
\overset{\sim}{\lrar} V_p$.

\end{itemize}
\end{enumerate}
\end{thm}

\begin{rem}
Since $\tr\, \phi$ and $\det \phi$ are globally
defined meromorphic sections of
$K_C$ and $K_C^{\tensor 2}$, 
respectively, the existence of
the Rees $D$-module $\widetilde{\cM}$ is
obvious. We can simply 
\emph{define} it by 
\eqref{qc main} and \eqref{local m}.
Therefore, the point here is that the 
differential operator $P(x,\hbar)$ has
a particular solution that is prescribed in 
the main theorem.
\end{rem}

\subsection{The topological recursion and the WKB 
method}
\label{sub:TR}

Let us start with defining each terminology 
in the main theorem.

Although the topological recursion can 
be formulated for an arbitrary ramified
covering of a base curve $C$ of any degree, 
for the purpose of quantization in this paper, we need a
Galois covering, and we also need to 
calculate the residues in the formula. 
Therefore, we deal with the 
topological recursion only for a 
 covering of degree $2$ in this paper.

\begin{Def}[Integral topological 
recursion for a degree $2$ 
covering]
Let $C$ be a non-singular projective algebraic
curve, and $\tilde{\pi}:\widetilde{\Sigma}\lrar C$ a
degree $2$ covering by another non-singular
curve $\widetilde{\Sigma}$. We denote by 
$R$ the
ramification divisor of $\tilde{\pi}$. In this
case the covering
$\tilde{\pi}$ is a Galois covering with
the Galois group 
$\bZ/2\bZ = \la \tilde{\sigma} \ra$,
and $R$ is the fixed-point divisor of the involution
$\tilde{\sigma}$.
The \textbf{integral topological recursion}
is an inductive 
mechanism of constructing meromorphic
differential forms $W_{g,n}$ on the Hilbert scheme 
$\widetilde{\Sigma}^{[n]}$
of  $n$-points on $\widetilde{\Sigma}$
for all $g\ge 0$ and 
$n\ge 1$ in the \emph{stable range}
$2g-2+n>0$, from  given initial data $W_{0,1}$
and $W_{0,2}$. 

\begin{itemize}
\item $W_{0,1}$ is a meromorphic $1$-form
on $\widetilde{\Sigma}$.

\item $W_{0,2}$ is defined to be
\begin{equation}
\label{W02}
W_{0,2}(z_1,z_2) = d_1d_2 
\log E_{\widetilde{\Sigma}}(z_1,z_2),
\end{equation}
where $E_{\widetilde{\Sigma}}(z_1,z_2)$ is 
the normalized Riemann prime
form on 
$\widetilde{\Sigma}\times \widetilde{\Sigma}$
(see \cite[Section~2]{DM2014}).
\end{itemize}
Let $\omega^{a-b}(z)$ be a normalized
Cauchy kernel on $\widetilde{\Sigma}$,
which has simple poles at $z=a$ of residue $1$ 
and at $z=b$ of residue $-1$. 
Then (see \cite[Section~2]{DM2014})
$$
d_1 \omega^{z_1-b}(z_2) = W_{0,2}(z_1,z_2).
$$
Define
\begin{equation}
\label{Omega}
\Omega := \tilde{\sigma}^*W_{0,1}-W_{0,1}.
\end{equation}
Then $\tilde{\sigma}^*\Omega = -\Omega$, hence
$\supp (R)\subset \supp(\Omega)$,
where $\supp(\Omega)$ denotes the support of both
zero and pole divisors of $\Omega$.
The inductive formula of the topological 
recursion is then given by the following:
\begin{multline}
\label{integral TR}
W_{g,n}(z_1,\dots,z_n)
=\half \frac{1}{2\pi \sqrt{-1}} 
\sum_{p\in \supp(\Omega)}
\oint_{\gam_p}\frac{\omega^{\tilde{z}-z}(z_1)}
{\Omega(z)}\\
\times
\left[W_{g-1,n+1}(z,\tilde{z}, z_2,\dots,z_n)
+\sum_{\substack{g_1+g_2=g\\
I\sqcup J=\{2,\dots,n\}}}^{\rm{No }\; (0,1)}
W_{g_1,|I|+1}(z,z_I)W_{g_2,|J|+1}(\tilde{z},z_J)
\right].
\end{multline}
Here, 
\begin{itemize}
\item $\gam_p$ is a positively oriented small 
loop around a  point $p\in \supp (\Omega)$;
\item the integration is taken with respect to  
 $z\in \gam_p$ for each $p\in \supp (\Omega)$;
\item $\tilde{z} = \tilde{\sigma}(z)$ is the Galois
conjugate of $z\in \widetilde{\Sigma}$;
\item the operation 
$1/\Omega$ denotes the contraction of the
meromorphic vector field dual to the $1$-form
$\Omega$, considered
as a meromorphic section of
$K_{\widetilde{\Sigma}}^{-1}$;
\item ``No $(0,1)$'' means that
$g_1=0$ and $I=\emptyset$, or $g_2=0$ and
$J=\emptyset$, are excluded in the summation;
\item the sum runs over all partitions of $g$ and
set partitions of $\{2,\dots,n\}$, other than those 
containing the $(0,1)$ geometry; 
\item $|I|$ is the cardinality of the subset 
$I\subset \{2,\dots,n\}$; and
\item $z_I=(z_i)_{i\in I}$.
\end{itemize}
\end{Def}

The passage from the topological recursion 
\eqref{integral TR} to the quantum curve
\eqref{Sch} is the evaluation of the
residues in the formula. 

\begin{Def}[Free energies]
\label{def:free energy}
The \textbf{free energy} of type $(g,n)$ is
a function $F_{g,n}(z_1,\dots,z_n)$
defined on the universal covering
$\cU^n $ of $\widetilde{\Sigma}^n$ such that
$$
d_1\cdots d_n F_{g,n} = W_{g,n}.
$$
\end{Def}

\begin{rem}
The free energies may contain 
logarithmic singularities, since it is an integral
of a meromorphic function. 
For example, $F_{0,2}$ is the Riemann prime
form itself considered as a function on 
$\cU^2$, which has logarithmic singularities 
along the diagonal \cite[Section~2]{DM2014}.
\end{rem}

\begin{Def}[Differential  recursion for
a degree $2$ covering]
The \textbf{differential  
recursion} is the
following partial differential equation
for all $(g,n)$ subject to $2g-2+n\ge 2$:
\begin{multline}
\label{differential TR}
d_1 F_{g,n}(z_1,\dots,z_n)
\\
= \sum_{j=2}^n 
\left[
\frac{\omega^{z_j-\sigma(z_j)}
(z_1)}{\Omega(z_1)}\cdot 
d_1F_{g,n-1}\big(z_{[\hat{j}]}\big)
-
\frac{\omega^{z_j-\sigma(z_j)}
(z_1)}{\Omega(z_j)}\cdot 
d_jF_{g,n-1}\big(z_{[\hat{1}]}\big)
\right]
\\
+\frac{1}{\Omega(z_1)}
d_{u_1}d_{u_2}
\left.
\left[F_{g-1,n+1}
\big(u_1,u_2,z_{[\hat{1}]}\big)
+\sum_{\substack{g_1+g_2=g\\
I\sqcup J=[\hat{1}]}}^{\text{stable}}
F_{g_1,|I|+1}(u_1,z_I)F_{g_2,|J|+1}(u_2,z_J)
\right]
\right|_{\substack{u_1=z_1\\u_2=z_1}}.
\end{multline}
Here, $1/\Omega$ is again the contraction operation,
and the  index subset $[\hat{j}]$ denotes
the exclusion of $j\in \{1,2,\dots,n\}$.
\end{Def}

\begin{rem}
As pointed out in \cite[Remark 4.8]{DM2014}, 
\eqref{differential TR} is a coordinate-free 
equation, written in terms of exterior differentiations
and the contraction operation
on the universal covering of $\widetilde{\Sigma}$.
\end{rem}

\begin{thm}
Let $\varpi: \cU\lrar \widetilde{\Sigma}$
 be the universal covering of
$\widetilde{\Sigma}$. 
Suppose that $F_{g,n}$ 
for $2g-2+n>0$ are  globally meromorphic
on $\cU^{[n]}$ with poles 
located only along the divisor of 
$\cU^{[n]}$ when one of the factors 
lies in the pull-back
divisor $\varpi^*(\Omega)_0$ of zeros of 
$\Omega$.  Define
$W_{g,n} := d_1\cdots d_n F_{g,n}$. If 
$F_{g,n}$'s satisfy the differential  
recursion \eqref{differential TR}, then 
$W_{g,n}$'s satisfy the integral topological 
recursion \eqref{integral TR}. 
\end{thm}

Although the context of the statement is slightly
different, the proof is essentially the same as that
of
\cite[Theorem~4.7]{DM2014}.

Now let us consider a spectral curve 
$\Sigma\subset 
\overline{T^*C}$
 of \eqref{spectral curve}
defined by a pair of meromorphic sections
$a_1=-\tr\phi$ of $K_C$ and 
$a_2=\det\phi$ of $K_C^{\tensor 2}$.
Let
$\widetilde{\Sigma}$ be the desingularization of 
$\Sigma$ in \eqref{blow-up}.
We  apply the topological recursion
 \eqref{integral TR}
to the covering $\tilde{\pi}:\widetilde{\Sigma}\lrar C$
of \eqref{pi-tilde}. The geometry of the 
spectral curve $\Sigma$ provides us with a
canonical choice of the initial differential forms
\eqref{W0102}. 
At this point we pay a special attention that 
the topological recursions \eqref{integral TR}
and \eqref{differential TR} are both defined 
on the spectral curve 
$\widetilde{\Sigma}$, while we wish to 
construct a Rees $D$-module on $C$. 
Since the free energies are defined on the
universal covering of $\widetilde{\Sigma}$,
we need to have a mechanism to relate a
coordinate on the desingularized spectral curve
and that of the base curve $C$.

Take an arbitrary point $p\in C\setminus 
\supp(\Delta)$, and a local coordinate $x$ around 
$p$. 
Here, $\Delta$ is the discriminant divisor 
\eqref{discriminant}. By choosing a small 
disc $V$ 
around $p$, we can make the inverse image of 
$\tilde{\pi}:\widetilde{\Sigma}\lrar C$ 
consisting of two isomorphic discs. Since $V$ is
away from the critical values of $\tilde{\pi}$,
the inverse image consists of two discs in 
the original spectral curve $\Sigma$. 
Note that we choose an eigenvalue $\a$ of
$\phi$ on $V$ in Theorem~\ref{thm:main}. 
We are thus specifying one of the inverse image discs
here. Let us name the disc $V_\a$ that
corresponds to 
$\a$.

At this point apply the WKB analysis to the differential
equation \eqref{PPsi alpha} with 
 the WKB expansion of the solution
\begin{equation}
\label{WKBa}
\Psi^\a (x,\hbar) = \exp\left(\sum_{m=0}^\infty
\hbar^{m-1} S_m\big(x(z)\big)\right) = 
\exp F^\a(x,\hbar),
\end{equation}
where we choose a coordinate $z$ of $V_\a$
so that the function $x = x(z)$ represents the 
projection $\pi:V_\a \lrar V$.
The equation $P\Psi^\a = P e^{F^\a} = 0$ reads
\begin{equation}
\label{F alpha}
\hbar^2 \frac{d^2}{dx^2}F^\a + 
\hbar^2\frac{dF^\a}{dx}
\frac{dF^\a}{dx} +a_1 \hbar \frac{dF^\a}{dx} +a_2=0.
\end{equation}
The $\hbar$-expansion of \eqref{F alpha} 
gives
\begin{align}
\label{scl}
&\hbar^0{\text{-terms}}:
\quad (S_0'(x))^2 +a_1S_0'(x) + a_2=0,
\\
\label{consistency}
&\hbar^1{\text{-terms}}:\quad
2S_0'(x)S_1'(x) + S_0''(x)+a_1S_1'(x)=0,
\\
\label{h m+1}
&\hbar^{m+1}{\text{-terms}}:\quad
S_m''(x) +\sum_{a+b=m+1}
S_a'(x)S_b'(x)+a_1S_{m+1}'(x)=0, \quad m\ge 1,
\end{align}
where $'$ denotes the $x$-derivative.
The WKB method is to solve these equations
iteratively and find $S_m(x)$ for all $m\ge 0$. Here,
\eqref{scl} is the
\textbf{semi-classical limit}
of \eqref{PPsi alpha}, and 
\eqref{consistency} is the 
\emph{consistency condition} we need to solve
the WKB expansion.
Since the $1$-form $dS_0(x)$ is a local section
of $T^*C$, we identify $y=S_0'(x)$. Then 
\eqref{scl} is the local expression of the 
spectral curve equation 
\eqref{spectral curve general}.
This expression is the same everywhere
for $p\in C\setminus \supp(\Delta)$. We note 
$a_1$ and $a_2$ are globally defined.
Therefore, we
recover the spectral curve $\Sigma$ from 
the differential operator \eqref{qc main}.

The topological recursion provides a closed
formula for each $S_m(x)$.

\begin{thm}[Topological recursion and WKB]
Let us determine $S_0(x)$ and $S_1(x)$ from 
the semi-classical limit and the consistency condition.
Then the \textbf{principal specialization}
of \eqref{differential TR} is 
equivalent to \eqref{h m+1}.
\end{thm}

\begin{proof}
First let us take
 $q\in C$ one of the $q_i$'s of
\eqref{Delta 0}, above which 
$\tilde{\pi}:\widetilde{\Sigma}\lrar C$ is 
simply ramified 
at $Q:=\tilde{\pi}^{-1}(q)
\in \widetilde{\Sigma}$. We choose
a local coordinate $x$ on $C$ centered
at $q$. The Galois action 
of $\tilde{\sigma}$ on $\widetilde{\Sigma}$
fixes $Q$. Let $\widetilde{U}\subset 
\widetilde{\Sigma}$ be a  neighborhood of $Q$
such that $p\in \tilde{\pi}(\widetilde{U})$ 
and  $a_1\in H^0\big(\tilde{\pi}(\widetilde{U}),
K_C\big)$, i.e., holomorphic
on  $\tilde{\pi}(\widetilde{U})$. 
The defining equation of 
the spectral curve $\Sigma$ on 
$\tilde{\pi}(\widetilde{U})$ is
$$
\left(ydx+\half a_1\right)^2-
\left(\frac{1}{4}a_1^2-a_2\right)=0.
$$
Since $a_1$ is holomorphic at $x=q$, 
 the Galois action of $\sigma$ on the
spectral curve $\Sigma$ extends
to $\overline{T^*C}|_{\tilde{\pi}(\widetilde{U})}$
by the formula given in \eqref{involution}.
As we have shown in Case~1 of the proof of 
Theorem~\ref{thm:geometric genus formula},
 the degree $m_i$ of zero of the discriminant 
 $\frac{1}{4}a_1^2-a_2$ at $q=q_i$ is odd, say
 $m_i = 2\mu+1$, and the construction
 of $Bl_{\min}(\overline{T^*C})$ contains blow-ups of
 $\lfloor \frac{2\mu+1}{2}\rfloor = \mu$ times
 at the singular point above $q$. 
 In terms of the coordinate $x$, we can write 
 $$
 a_1=a_1(x)dx, \quad a_2=a_2(x)(dx)^2,\quad
 \frac{1}{4}a_1(x)^2-a_2(x) = c x^{2\mu+1} 
 $$
 with a unit $c\in \cO_{C,q}$.
 Define $y_0:=y+\half a_1(x)$. 
 Then the first blow-up at the singular point above
 $q$ is done by replacing $y_0 = y_1x$ so that 
 the proper transform is locally defined by 
 $$
 y_1^2 = cx^{2\mu-1}.
 $$
 The coordinate $y_1$ is the affine coordinate
 of the exceptional divisor.
 Repeating this process $\mu$-times, we end up with
 a coordinate $y_{\mu-1} = y_\mu x$ and
 an equation
 $$
 y_\mu ^2 = c x.
 $$
 Here again, $y_\mu$ is the affine coordinate of
 the last exceptional divisor resulted from the
 $\mu$-th blow-up. 
 We now write $z=y_\mu$ so that the proper transform
 of the $\mu$-times blow-ups is given by
 \begin{equation}
 \label{z and x}
 z^2 = c x.
 \end{equation}
 Note that the Galois action of $\tilde{\sigma}$
 at $Q$ is simply $z\longmapsto -z$.
 Solving \eqref{z and x} as a functional equation,
 we obtain a Galois invariant local expression
 \begin{equation}
 \label{x(z)}
 x = x(z) = c_Q(z^2) z^2,
 \end{equation}
 where $c_Q\in \cO_{\widetilde{\Sigma},Q}$
 is a unit element. This formula \eqref{x(z)}
 is precisely the local expression of the 
 morphism $\tilde{\pi}:\overline{\Sigma}
 \lrar C$ at $Q\in \overline{\Sigma}$.
 On the other hand, from the construction we 
 also have
 $$
 y_0 = z x^\mu = y+\half a_1(x),
 $$
 or equivalently,
 \begin{equation}
 \label{y(z)}
 \eta = y(z)dx = zx^\mu dx - \half a_1, \qquad 
 y(z)= zx^\mu  - \half a_1(x).
 \end{equation}
 We have thus obtained the normalization 
 coordinate $z$ on the desingularized 
 curve $\overline{\Sigma}$ near $Q$:
 \begin{equation}
 \label{normalization coordinate}
 \begin{cases}
 x = x(z) = c_Q(z^2) z^2,\\
 y = y(z) = zx^\mu  - \half a_1(x).
 \end{cases}
 \end{equation}
 Notice that we now have a parametric 
 equation for the singular spectral curve $\Sigma$:
 $$
\left( y(z) +\half a_1(x(z))\right)^2 
= z^2 x(z)^{2\mu}
= c x(z)^{2\mu+1} 
= \frac{1}{4}a_1(x(z))^2 - a_2(x(z)).
 $$
 The differential form $\eta$ of \eqref{y(z)}
 is the local expression of the form 
 $\Omega$ in the differential topological 
 recursion \eqref{differential TR}.
 
 We have now established the local expression of
 all functions and forms involved in the 
 topological recursion. From here the rest of the 
 proof is parallel to \cite{DM2014}.
 
 As we have shown in the process of the
 proof of 
 Theorem~\ref{thm:geometric genus formula},
 the situation is the same if $q\in C$ corresponds
 to a branch point of $\tilde{\pi}:\widetilde{\Sigma}
 \lrar C$ that comes from an odd cusp 
 of $\Sigma$ on the divisor $C_\infty$.
 A similar argument of the above proof 
 works for this case.

  This completes the proof.
\end{proof}

We have thus completed the proof of
Theorem~\ref{thm:main}.

\subsection{Singularity of quantum curves}
\label{sub:quantum singularity}

Let $p$ be a pole of the discriminant divisor
$\Delta$ of \eqref{Delta infinity}.
The local equation for the spectral curve 
around $p$ is 
$$
y^2+a_1(x)y+a_2(x)=0.
$$
As we have shown above,
the local generator of the quantum curve
as a Rees $D$-module is given by
a differential operator 
$$
\left(\hbar \frac{d}{dx}\right)^2
+a_1(x) \left(\hbar \frac{d}{dx}\right) + a_2(x).
$$
Therefore, the type of the singularity of the
quantum curve is determined by the local 
geometry of the spectral curve. 
We have  the following.

\begin{thm}[Regular and irregular singular points
of the quantum curve]
\label{thm:qc singularity}
Let $P\in \Sigma\cap C_\infty$ be a point at 
the intersection of 
the spectral curve $\Sigma$ and the divisor
$C_\infty$ at infinity of the ruled surface
$\overline{T^*C}$. 
Suppose it requires $\rho$ times blow
up at $P$ to construct $Bl(\overline{T^*C})$. 
Then the quantum curve of 
Theorem~\ref{thm:main}
has
\begin{itemize}
\item
 a regular singular point  at $p=\pi(P)$ if
 $\rho =1$.
 
\item If $\rho >1$, then the quantum curve has
an irregular singular point at $p=\pi(P)$
of class either $\rho-1$ or $\rho -\frac{3}{2}$,
the latter occurring only when $P$ is a cusp
singularity of $\Sigma$.
\end{itemize}
\end{thm}

\begin{proof}
As in the proof of 
Theorem~\ref{thm:geometric genus formula},
we denote by $k$ (reps.\ $\ell$)
the pole order of $a_1(x)$ (reps.\ $a_2(x)$) at 
$x=p$. Let $r$ be the invariant defined in 
\eqref{r}. Then by 
Definition~\ref{def:regular and irregular},
$p$ is a regular singular point of the
quantum curve if $0<r\le 1$. In this
case we need to blow-up once at $P$ for
construction of $Bl(\overline{T^*C})$, 
because 
$\lceil r \rceil = 1$.
If $r>1$, then the singularity is irregular with 
class $r-1$, and we need $\lceil r \rceil$ times
blow-ups. As we see from  the proof
of Theorem~\ref{thm:nonsingular}, a
non-integer $r$ occurs only when
$P$ is a cusp. 
This completes the proof.
\end{proof}

\section{The classical differential equations
as quantum curves}
\label{sect:classical}

The key examples of the theory of quantum 
curves as presented in this paper are the
classical differential equations. In this section,
we present the Hermite and Gauss hypergeometric
differential equations.

\subsection{Hermite differential equation}
\label{sub:Hermite}

The base curve is   $C=\bP^1$, 
as in the Airy case.
The  stable Higgs bundle $(E,\phi)$ consists of
 the trivial vector bundle
$
E=\cO_{\bP^1}\dsum \cO_{\bP^1}
$
 and a Higgs field
\begin{equation}
\label{phi-Hermite}
\phi = 
\begin{bmatrix}
& 1\\-1\ &-x
\end{bmatrix} dx
:
E\lrar E
\tensor K_{\bP^1}(2) = E.
\end{equation}
In the affine coordinate $(x,y)$ of the Hirzebruch
surface $\bF^2$, 
the spectral curve $\Sigma$ is given by
\begin{equation}
\label{Hermite-spectral}
\det\big(\eta-\pi^*(\phi)\big) = (y^2 +xy +1)(dx)^2 = 0,
\end{equation}
where $\pi:\bF_2\lrar \bP^1$ is the projection.
In the other affine coordinate
$(u,w)$ of \eqref{uw}, 
the spectral curve is singular at $(u,w) = (0,0)$:
\begin{equation}
\label{Hermite singularity}
u^4  -uw +w^2 = 0.
\end{equation}
These equations tell us that 
$\Sigma \cdot C_0 = 0$ and $\Sigma \cdot 
C_\infty = 4$. Therefore, 
$$
\Sigma = 2C_0+4F \in \NS(\bF_2).
$$

\begin{figure}[htb]
\centerline{\epsfig{file=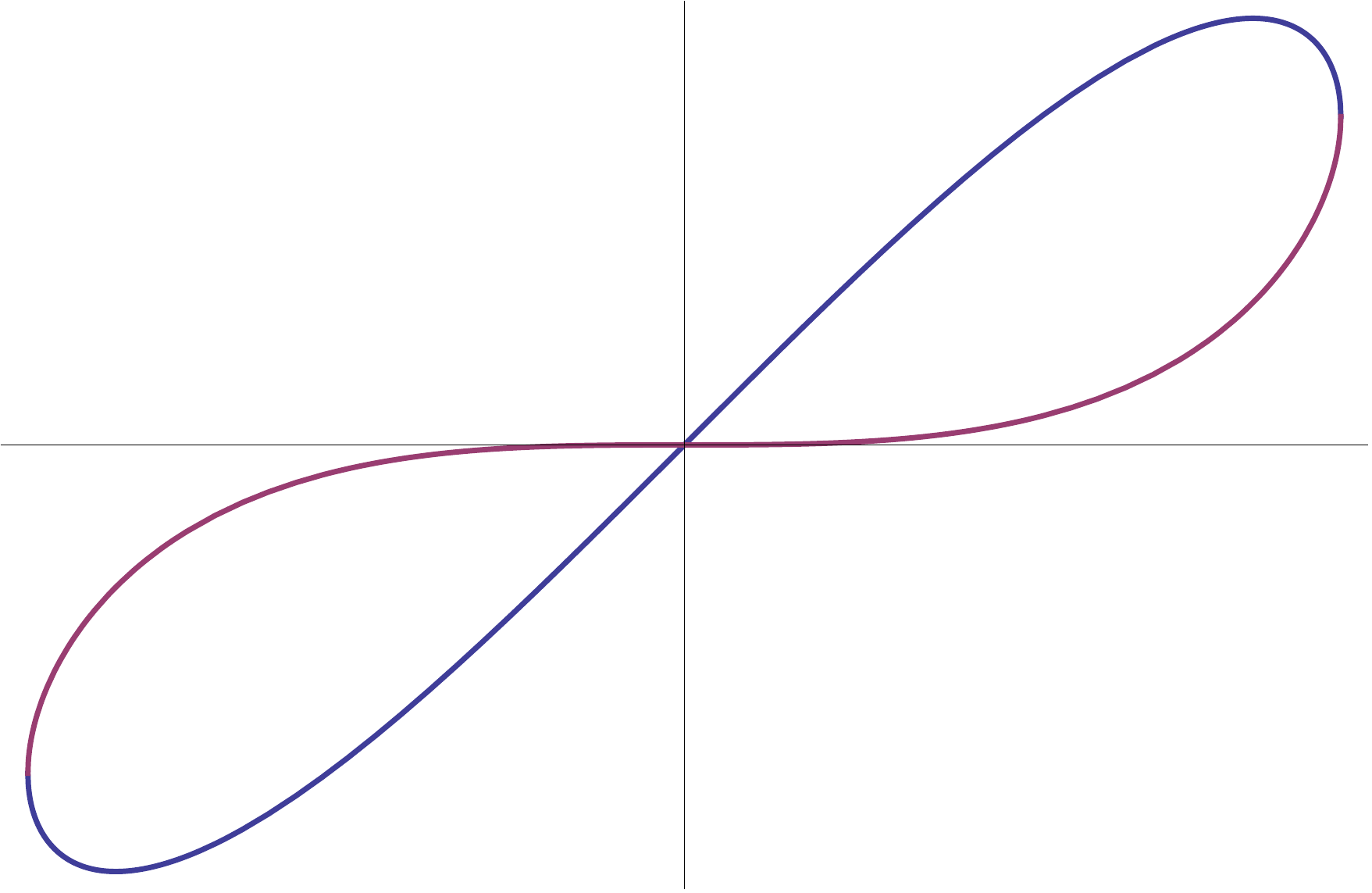, width=2in}}
\caption{The spectral curve $\Sigma$ of
\eqref{Hermite-spectral}. The horizontal 
line is the divisor $C_\infty$ at infinity, and the vertical
line is the fiber class $F$. The spectral curve
intersects with $C_\infty$ four times. One of the two
curve germ components 
 is given by $w=u$, and
the other by $w=u^3$.}
\label{fig:Hermite spectral}
\end{figure}

\begin{figure}[htb]
\centerline{\epsfig{file=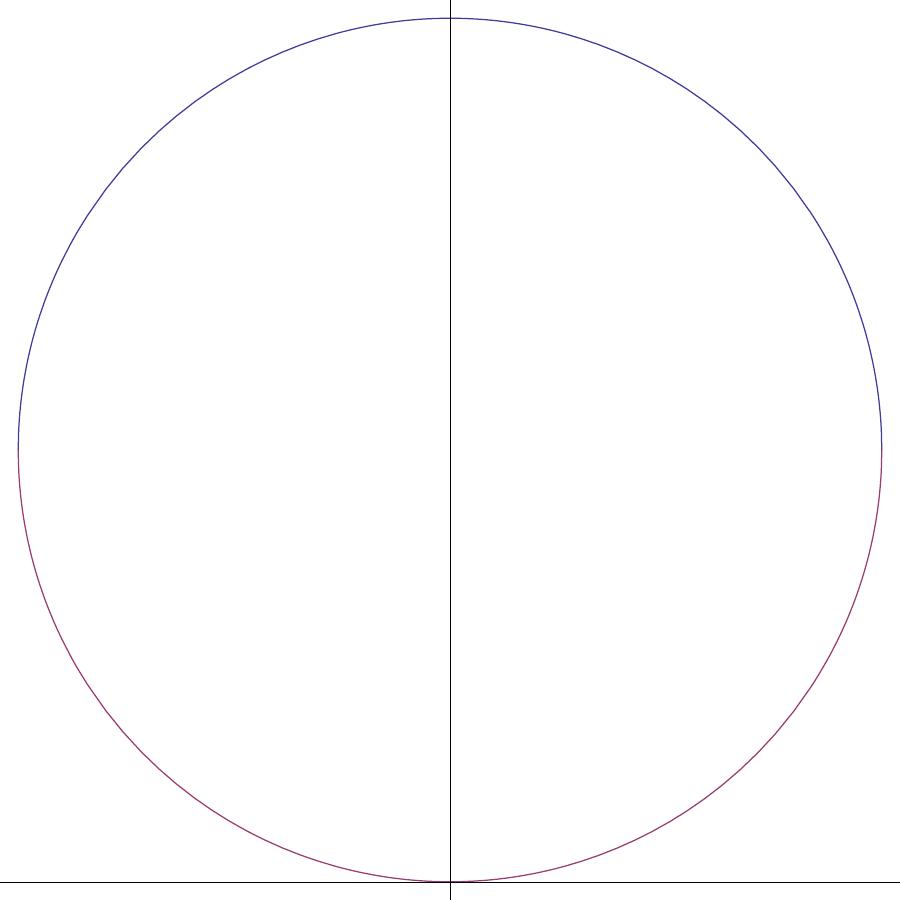, width=1in}}
\caption{The desingularization $\widetilde{\Sigma}$ of
the spectral curve
\eqref{Hermite-spectral}. }
\label{fig:Hermite spectral 2}
\end{figure}

The discriminant of the defining equation
\eqref{Hermite-spectral}
is
$$
\left(-\frac{1}{4}x^2+1\right)(dx)^2
= -\frac{1}{4} (x-2)(x+2)(dx)^2
=\frac{u^2-\frac{1}{4}}{u^6}(du)^2.
$$
It has two simple zeros at $x=\pm 2$ and
a pole of order $6$ at $x=\infty$. 
We note that 
$$
\tr(\phi) = -xdx= \frac{du}{u^3}
$$
has a cubic pole at $u=0$.
As explained in Case~3 of the proof of
Theorem~\ref{thm:geometric genus formula},
we need to compare the poles of $\tr(\phi)$ and
$$
\det(\phi) = (dx)^2 = \frac{(du)^2}{u^4}.
$$
Since $4-3=1$,
we blow up $\Sigma$ once at its nodal singularity
 $(u,w)=(0,0)$.
 We introduce $w=w_1 u$. Then \eqref{Hermite singularity} becomes
 \begin{equation}
 \label{circle}
 u^2 + \left(w_1-\half\right)^2 = \frac{1}{4}.
 \end{equation}
The geometric 
genus formula 
\eqref{pg}
tells us that $\Sigma_{\min}$
has genus $0$, and 
$\Sigma_{\min}\lrar C$
is ramified at two points, 
corresponding to the original ramification 
points $(x,y) = (\pm 2,\mp1)$ of $\Sigma$.
The rational parametrization of \eqref{circle} is
given by 
$$
\begin{cases}
u =\half \cdot \frac{t^2-1}{t^2+1}\\
w_1=\half -\frac{t}{t^2+1},
\end{cases}
$$
where $t$ is an affine coordinate of 
$\widetilde{\Sigma}$ such that $t=\pm 1$ gives
$(u,w)= (0,0)$. The parameter $t$ is a normalization
coordinate of the spectral curve $\Sigma$:
\begin{equation}
\label{Hermite t}
\begin{cases}
x = 2+\frac{4}{t^2+1}\\
y=-\frac{t+1}{t-1},
\end{cases}
\qquad
\begin{cases}
u = \half \cdot \frac{t^2-1}{t^2+1}\\
w=\frac{1}{4}\cdot \frac{(t-1)^3(t+1)}{(t^2+1)^2}.
\end{cases}
\end{equation}
We notice that
the  expression of \eqref{Hermite t} is exactly
the same as \cite[(3.13), (3.14)]{DMSS}.
The integral topological recursion 
applied to  $\widetilde{\Sigma}$
again agrees with that of \cite{DMSS}.

The quantum curve construction of \cite{MS}
is thus consistent with our new definition. The 
result is
\begin{equation}
\label{qc-Hermite}
\left(
\left(\hbar \frac{d}{dx}\right)^2 +
x \left(\hbar \frac{d}{dx}\right) +1
\right)\Psi(x,\hbar) = 0.
\end{equation}
Since $x$ and $d/dx$ do not commute, the
passage from \eqref{Hermite-spectral} to
\eqref{qc-Hermite} is non-trivial, in the sense
that the constant term could have contained
a term $c\cdot \hbar$. On the affine open subset
$U_1=\bP^1\setminus \{0\}$, the operator
of \eqref{qc-Hermite}
has an expression
$$
u^4 \left(\hbar \frac{d}{du}\right)^2
+(2u^3\hbar -u)\left(\hbar \frac{d}{du}\right) +1
\in \widetilde{\cD_{U_1}}.
$$
Thus the point $\infty\in \bP^1$ is  an irregular
singular point of class $2$ of \eqref{qc-Hermite}
for $\hbar\ne 0$.

The semi-classical limit
\eqref{scl} of 
\eqref{qc-Hermite} using the WKB formula
\eqref{WKB} is
\begin{equation}
\label{Catalan scl}
S_0'(x)^2 +x S_0'(x) + 1 = 0.
\end{equation}
Following \cite{DMSS}, define
\begin{equation}
\label{z}
z = \frac{t+1}{t-1} = \sum_{m=0}^\infty
\frac{C_m}{x^{2m+1}},
\end{equation}
where $C_m=\frac{1}{m+1}\binom{2m}{m}$
is the $m$-th Catalan number. 
The inverse function of \eqref{z} for 
$x= \infty \Rightarrow z=0$ is given by
$x = x(z) = z+ \frac{1}{z}$. In terms of $z$,
the two solutions of \eqref{Catalan scl}
are given by 
\begin{equation}
\label{Catalan S0}
S_0\big(x(z)\big) =
\begin{cases}
 -\half z^2 +\log z +\const
 \\
 -\half \frac{1}{z^2} -\log z +\const.
 \end{cases}
\end{equation}
Corresponding to these choices,
the solutions to the consistency condition 
\eqref{consistency} are given by
\begin{equation}
\label{Catalan S1}
S_1\big(x(z)\big) = 
\begin{cases}
-\half \log(1-z^2)+\const
\\
-\half \log(1-z^2)+\log z+\const.
\end{cases}
\end{equation}

Every solution of \eqref{qc-Hermite}
is a linear combination of two solutions,
with coefficients given by
arbitrary functions in $\hbar$. One is
 given by
the \textbf{Kummer confluent 
hypergeomtric function} of  \eqref{Kummer}:
\begin{equation*}
\Psi_1(x,\hbar) = {}_1F_1\left(\frac{1}{2\hbar};
\half;-\frac{x^2}{2\hbar}\right).
\end{equation*}
The other is a bit more complicated function known
as the \textbf{Tricomi confluent 
hypergeomtric function} 
\begin{equation}
\label{Tricomi}
\Psi_2(x,\hbar) = 
\frac{\Gamma[\half]}{\Gamma[\frac{1}{2\hbar}
+\half]}
{}_1F_1\left(\frac{1}{2\hbar};
\half;-\frac{x^2}{2\hbar}\right)
+\frac{\Gamma[-\half]}{\Gamma[\frac{1}{2\hbar}]}
\sqrt{\frac{x^2}{-2\hbar}}
{}_1F_1\left(\frac{1}{2\hbar}
+\half;
\frac{3}{2};-\frac{x^2}{2\hbar}\right).
\end{equation}
For a positive real $\hbar >0$, 
let us consider a special solution 
\begin{equation}
\label{eq:Catalan Psi}
\Psi^{\Catalan}(x,\hbar):=
\left(-\frac{1}{2\hbar}\right)^{\frac{1}{2\hbar}}
\Psi_2(x,\hbar).
\end{equation}
This solution corresponds to the 
WKB solution \eqref{WKB} for the 
first choices of \eqref{Catalan S0}
and \eqref{Catalan S1}, with both constants
of integration to be set $0$.
Then we have a closed formula for the 
all-order asymptotics of this particular confluent
hypergeometric function:
\begin{equation}
\begin{aligned}
\label{Catalan Psi expansion}
\Psi^{\Catalan}(x,\hbar)
&=
\left(\frac{1}{x}\right)^{\frac{1}{\hbar}}
\sum_{n=0}^\infty 
\frac{\hbar^n \left(\frac{1}{\hbar}\right)_{2n}}
{(2n)!!}\cdot  \frac{1}{x^{2n}}
\\
&=
\exp\left(
\sum_{2g-2+n\ge -1}\frac{1}{n!}\hbar^{2g-2+n}
F_{g,n}^\Catalan(x,\dots,x)
\right),
\end{aligned}
\end{equation}
where $(1/\hbar)_{2n}$ is the Pochhammer
symbol \eqref{Poch}.
The free energies are defined by
\begin{equation}
\label{Fgn Catalan}
F_{g,n}^\Catalan(x_1,\dots,x_n)
=
\sum_{\mu_1,\dots,\mu_n>0}
\frac{C_{g,n}(\mu_1,\dots,\mu_n)}
{\mu_1\cdots\mu_n}
\prod_{i=1}^n x_i^{-\mu_i}
\end{equation}
for $2g-2+n>0$, and $F_{0,1}^\Catalan (x)= S_0(x)$
and $\half F_{0,2}^\Catalan (x)= S_1(x)$.
Here, $C_{g,n}(\mu_1,\dots,\mu_n)$ is
the generalized Catalan number of genus $g$ and
$n$ labeled vertices of degrees
$(\mu_1,\dots,\mu_n)$ that counts the number
of \textbf{cellular graphs}
\cite{DMSS, WL}. In 
\cite[Theorem~4.3, Proposition A.1]{DMSS},
we show that $F_{g,n}^C$
satisfies the differential recursion equation 
\eqref{differential TR}. (We note that
the differential recursion of \cite{DMSS} is
derived by taking the Laplace transform 
of \cite[Equation 6]{WL}). The initial geometric
data \eqref{W0102} are also the same
as \cite[(3.12), (4.3)]{DMSS}. Therefore,
the application of the topological recursion
to the desingularized spectral curve 
$\Sigma_{\min}$ produces 
\eqref{Fgn Catalan}, and the quantum curve
\eqref{qc-Hermite}.

At the special value $\hbar = 1$, 
the expansion \eqref{Catalan Psi expansion} has
the following simple form
$$
\sum_{n=0}^\infty
\frac{(2n-1)!!}{x^{2n+1}}=\exp\left(
\sum_{2g-2+n\ge -1}\frac{1}{n!}
\sum_{\mu_1,\dots,\mu_n>0}
\frac{C_{g,n}(\mu_1,\dots,\mu_n)}
{\mu_1\cdots\mu_n}
\prod_{i=1}^n x^{-(\mu_1+\cdots +\mu_n)}
\right).
$$
Note that the sum of the degrees of the vertices
$\mu_1+\cdots +\mu_n$ is always even. Therefore,
except for the unstable geometries $(g,n)=(0,1)$
and $(0,2)$, 
the above expansion is in $x^{-2}$. This indicates
 that the Hermite equation has
 an irregular singular point of class 
$2$  at 
$x=\infty$.

\subsection{Gau\ss\ hypergeometric differential
equation}
\label{sub:Gauss}

The Higgs bundle $(E,\phi)$ on $\bP^1$ is
again given by the trivial bundle 
$E=\cO_{\bP^1}\dsum \cO_{\bP^1}$, and a 
Higgs field
 \begin{equation}
 \label{Gauss phi}
\phi= \begin{bmatrix}
&\frac{1}{x}\\ \\
-\frac{ab}{x-1}&-\frac{(a+b+1)x-c}{x(x-1)}
\end{bmatrix}dx,
\end{equation}
where $a,b,c$ are constant parameters.
The spectral curve $\Sigma\in \bF_2$ is defined 
by
\begin{equation}
\label{Gauss spectral}
x(x-1)y^2
+((a+b+1)x-c)y
+ab=0.
\end{equation}
In terms of the $(u,w)$ coordinate of
\eqref{uw}, the spectral curve is given by
\begin{equation}
\label{Gauss singularity}
abw^2+(cu-a-b)uw
-u^2(u-1)=0.
\end{equation}
It has an ordinary double point at $(u,w)=(0,0)$.
The discriminant divisor
of  \eqref{discriminant} is
\be
\label{Gauss discriminant}
\Delta^\Gauss = 
\left(\frac{\left(\frac{1}{4}((a+b+1)x-1)^2
-abx(x-1)\right)(dx)^2}
{x^2(x-1)^2}
\right),
\ee
which consists of two simple zeros
and $3$ double poles at $x=0, 1, \infty$. 
Following 
Definition~\ref{def:Bl},
we blow up $\bF_2$ once at the point at infinity
of $\Sigma$
to construct the normalization $\Sigma_{\min}$.
The invariant $\b$ of \eqref{delta}
is equal to $2$, and hence 
$\Sigma_{\min}$ is isomorphic to 
$\bP^1$.

The quantum curve we obtain is a
Gau\ss\ hypergeometric differential equation
\be
\label{hbar Gauss}
\left(x(x-1)\left(\hbar\frac{d}{dx}\right)^2
+\left((a+b+1)x-c\right)\hbar\frac{d}{dx}
+ab\right)\Psi^\Gauss (x,\hbar) = 0.
\ee
One of the two independent solutions
that is holomorphic 
at $x=0$ is given in terms of
a \textbf{Gau\ss\ hypergeometric function}
\begin{multline}
\label{psi Gauss}
\Psi^\Gauss (x,\hbar)=
  {}_2F_1\Bigg(-\frac{\sqrt{\left
   (a+b+1-\hbar\right)^2-4 a
   b}}{2 \hbar}+\frac{a+b+1}{2
   \hbar}-\frac{1}{2},
   \\
   \frac{\sqrt{\left
   (a+b+1-\hbar\right)^2-4 a
   b}}{2 \hbar}+\frac{a+b+1}{2
   \hbar}-\frac{1}{2};\frac{c}{\hbar};x\Bigg).
   \end{multline}
If we choose $\sqrt{\left
   (a+b+1-\hbar\right)^2-4 a
   b}=b-a$ when $\hbar = 1$, then 
   \be
   \label{psi Gauss 1}
   \Psi^\Gauss (x,1) = {}_2F_1(a,b;c,x)
   =\sum_{n=0}^\infty\frac{(a)_n(b)_n}{(c)_n}
   \;\frac{x^n}{n!}
   \ee
   solves the standard form of the
   Gau\ss\ hypergeometric equation
   \be
   \label{Gauss 1}
   \left(x(x-1)\left(\frac{d}{dx}\right)^2
+\left((a+b+1)x-c\right)\frac{d}{dx}
+ab\right)\Psi^\Gauss (x,1) = 0.
   \ee

Now let us specialize $a=b=\half, c=1$. 
We have the relation between the 
hypergeometric function and the
period function of \eqref{elliptic periods}:
\be
\label{period = 2F1}
\frac{\omega_1(x)}{\pi} = 
{}_2F_1\left(\half,\half;1,x\right).
\ee
The spectral curve \eqref{Gauss spectral}
   becomes 
\be
\label{Gauss spectral special}
x(x-1)y^2 +(2x-1)y +\frac{1}{4}=0.
\ee
   On the normalization
   $\widetilde{\Sigma}$, we have
   $W_{0,1}(x) = ydx$, which actually 
   depends of the
   sheet of the covering $\tilde{\pi}:
   \widetilde{\Sigma}\lrar \bP^1$. For our
   purpose, we choose 
   \be
   \label{y Gauss}
   y =y(x)= \frac{-(2x-1)-\sqrt{3x^2-3x+1}}
   {2 x\left(x-1\right)}.
   \ee
Then 
\begin{multline}
\label{S0 Gauss}
S_0(x) = F_{0,1}(x) = \int y(x)dx
\\
=
\frac{x}{4}-\frac{21 \left(4 \sqrt{3}-7\right)
   }{32 \left(2 \sqrt{3}-3\right)^2}x^2+\frac{23
   \left(26 \sqrt{3}-45\right) }{32 \left(2
   \sqrt{3}-3\right)^3}x^3-\frac{2547 \left(56
   \sqrt{3}-97\right) }{1024 \left(2
   \sqrt{3}-3\right)^4}x^4
   \\
   +\frac{7281 \left(362
   \sqrt{3}-627\right) }{2560 \left(2
   \sqrt{3}-3\right)^5}x^5
   -\frac{38115 \left(780
   \sqrt{3}-1351\right) }{4096 \left(2
   \sqrt{3}-3\right)^6}x^6
   +\frac{265869 \left(5042
   \sqrt{3}-8733\right) }{28672 \left(2
   \sqrt{3}-3\right)^7}x^7+\cdots
   \end{multline}
solves the semi-classical limit equation
\eqref{scl}. The solution of the 
consistency condition \eqref{consistency}
is given by
\begin{multline}
\label{S1 Gauss}
S_1(x) = -\int\frac{y'(x)}{2y(x)+
\frac{2x-1}{x(x-1)}}dx
\\
=
-\frac{7 }{32}x^2-\frac{53 }{96}x^3-\frac{1075
   }{1024}x^4-\frac{4319 }{2560}x^5-\frac{28319
   }{12288}x^6-\frac{72109
   }{28672}x^7+\cdots.
\end{multline}
The solution of   \eqref{h m+1} for $m=1$ is
\begin{equation}
\label{S2 Gauss}
S_2(x) = \frac{7 x^2}{32}+\frac{113 x^3}{96}+\frac{1821
   x^4}{512}+\frac{1269 x^5}{160}+\frac{56151
   x^6}{4096}+\frac{487323
   x^7}{28672}+\cdots.
\end{equation}
From \eqref{S0 Gauss}, \eqref{S1 Gauss}
and \eqref{S2 Gauss}, we have an expression
\begin{multline}
\label{Gauss asymptotic}
\exp\left(\frac{1}{\hbar}S_0(x)+S_1(x)+\hbar S_2(x)
\right)
\\
=
1+\frac{x}{4\hbar}+\frac{1+7\hbar-7\hbar^2
+7\hbar^3}{32\hbar^2}x^2
+\frac{1+21\hbar+71\hbar^2 - 191\hbar^3+\cdots}
{384\hbar^3}x^3
\\
+
\frac{1+42\hbar+473\hbar^2+598\hbar^3+\cdots}
{6144\hbar^4}x^4 
+
\frac{1+70\hbar+1585\hbar^2+1141-\hbar^3+\cdots}
{122880\hbar^5}x^5+\cdots.
\end{multline}
This is in good agreement of the hypergeometric
function of \eqref{psi Gauss}, which can be 
expanded as
\begin{multline*}
\Psi^\Gauss(x,\hbar)
= 1+\sum_{n=1}^\infty \frac{1}{4^n n!}
\frac{\prod_{m=1}^n\left(1+8(m-1)\hbar+
4(m-1)(m-2)\hbar^2\right)}
{\prod_{m=1}^n\left(1+(m-1)\hbar\right)}
\left(\frac{x}{\hbar}\right)^n
\\
=
1+\frac{x}{4\hbar}+\frac{1+8\hbar}{32\hbar
(1+\hbar)}x^2
+\frac{(1+8\hbar)(1+16\hbar+8\hbar^2)}
{384\hbar^3 (1+\hbar)(1+2\hbar) }x^3
\\
+\frac{(1+8\hbar)(1+16\hbar+8\hbar^2)
(1+24\hbar+24\hbar^2)}
{6144\hbar^4 (1+\hbar)(1+2\hbar)(1+3\hbar)
} x^4+\cdots,
\end{multline*}
up to order of $\hbar^3$ of the numerator of
every coefficient of $x^n$, $n\ge 0$. 
The topological recursion applied to the
spectral curve \eqref{Gauss spectral}
with $F_{0,1} = S_0(x)$ and the standard 
Riemann prime form on $\widetilde{\Sigma} = \bP^1$
for $F_{0,2}$ then gives a genus expansion of
$\Psi^\Gauss(x,\hbar)$, constructing a
genus $g$ B-model on the curve
\eqref{Gauss spectral}.

\hyphenation{Gar-ou-fal-i-dis}

\begin{ack}
The authors are grateful to  
the American Institute of Mathematics in Palo Alto, 
the Banff International Research Station,
 the Institute for Mathematical
Sciences at the National University of Singapore,
 Kobe University, and
 Max-Planck-Institut f\"ur Mathematik in Bonn,
for their hospitality and financial support. A large
portion of this work is carried out during the
authors' stay
in these institutions. 
They also thank
J\o rgen Andersen,
Philip Boalch, 
Leonid Chekhov,
Bertrand Eynard,
Tam\'as Hausel,
Kohei Iwaki,
Maxim Kontsevich,
Alexei Oblomkov,
Albert Schwarz, 
Yan Soibelman, 
 Ruifang Song,
 and Peter Zograf
for useful comments,   suggestions, and discussions.
M.M.\ thanks the Euler International 
Mathematical Institute in St.~Petersburg for 
hospitality, where the paper is completed.
The research of O.D.\ has been supported by
 GRK 1463 \emph{Analysis,
Geometry, and String Theory} at the 
Leibniz Universit\"at 
 Hannover.
The research of M.M.\ has been supported 
by MPIM in Bonn, 
NSF grants DMS-1104734 and DMS-1309298, 
and NSF-RNMS: Geometric Structures And 
Representation Varieties (GEAR Network, 
DMS-1107452, 1107263, 1107367).
\end{ack}


\providecommand{\bysame}{\leavevmode\hbox to3em{\hrulefill}\thinspace}

\bibliographystyle{amsplain}

\end{document}